\documentclass[reqno,12pt]{article}

\usepackage{a4wide}
\usepackage{amsmath} 
\usepackage{amssymb}
\usepackage{amsthm}
\usepackage[utf8]{inputenc} 
\usepackage{graphicx} 
\usepackage[english, russian]{babel}

\numberwithin{equation}{section}

\newtheorem{theo}{Theorem}

\newtheorem{coro}{Corollary}
\newtheorem{prop}{Proposition}

\newtheorem{lem}{Lemma}

\theoremstyle{remark}
\newtheorem{Remark}{Remark}

\newcommand{\N}{\mathbb{N}}
\newcommand{\Z}{\mathbb{Z}}
\newcommand{\Q}{\mathbb{Q}}
\newcommand{\R}{\mathbb{R}}
\renewcommand{\C}{\mathbb{C}}
\newcommand{\K}{\mathbb{K}}

\newcommand{\Qbar}{\overline{\mathbb Q}}
\newcommand{\Qb}{\overline{\mathbb Q}}

\newcommand{\etoile}{^\star}

\newcommand{\eps}{\varepsilon}

\newcommand{\Li}{{\rm Li}}

\newcommand{\OK}{{\mathcal O}_{\K}}

 \newcommand{\ddj}{\ddjj_j}
  \newcommand{\ddjj}{\chi}
 \newcommand{\ddjspu}{\chi_{s+1,j}}

\newcommand{\ppP}{\kappa} 
\newcommand{\qqQ}{K} 
\newcommand{\alcra}{\alpha_{p,j}}

\newcommand{\nvP}{\beta}
\newcommand{\nvU}{B}
\newcommand{\peu}{b}

\newcommand{\Lcor}{\mathbb{L}}

\makeatletter
\newcommand*{\house}[1]{%
  \mathord{%
    \mathpalette\@house{#1}%
  }%
}
\newcommand*{\@house}[2]{%
  \dimen@=\fontdimen8 %
      \ifx#1\scriptscriptstyle\scriptscriptfont
      \else\ifx#1\scriptstyle\scriptfont
      \else\textfont\fi\fi
      3 %
  \sbox0{%
    $#1%
      \vrule width\dimen@\relax
      \overline{%
        \kern2\dimen@
        \begingroup 
          #2%
        \endgroup
        \kern2\dimen@
      }%
      \vrule width\dimen@\relax
      \mathsurround=1.5\dimen@ 
    $%
  }%
  \ht0=\dimexpr\ht0-\dimen@\relax
  \dp0=\dimexpr\dp0+2\dimen@\relax
  \vbox{%
    \kern\dimen@ 
    \copy0 %
  }%
}

\newcommand{\uuu}{v}
\newcommand{\joliL}{{\mathcal L}}
\newcommand{\chemti}{\widetilde{\joliL}}
\newcommand{\chemhat}{\widehat{\joliL}}
\newcommand{\chemz}{\joliL_0}
\newcommand{\chemu}{\joliL_1}
\newcommand{\chemd}{\joliL_2}
\newcommand{\chemt}{\joliL_3}
\newcommand{\chemq}{\joliL_4}
\newcommand{\chemc}{\joliL_5}
\newcommand{\Span}{{\rm Span}}

\newcommand{\grosb}{{\mathcal B}_{S,r,n,j}(\alpha)}
\newcommand{\lz}{\ell_0}
\renewcommand{\Re}{{\rm Re}\, }
\renewcommand{\Im}{{\rm Im}\, }

\begin{document}

 \selectlanguage{english}

\title{Linear independence of values of $G$-functions}
\date\today
\author{S. Fischler and T. Rivoal}
\maketitle

\begin{abstract} Given any non-polynomial $G$-function $F(z)=\sum_{k=0}^\infty A_kz^k$ of radius of convergence $R$, we consider the $G$-functions   
$F_n^{[s]}(z)=\sum_{k=0}^\infty \frac{A_k}{(k+n)^s}z^k$ for any integers $s\ge 0$ and $n\ge 1$. For any fixed algebraic number $\alpha$ such that $0<\vert \alpha\vert<R$ and any number field $\K$ containing $\alpha$ and the $A_k$'s, we define $\Phi_{\alpha, S}$ as the $\K$-vector space generated by the values $F_n^{[s]}(\alpha)$, $n\ge 1$ and $0\le s\le S$. We prove that $u_{\K,F}\log(S)\le \dim_\K(\Phi_{\alpha, S })\le v_FS$  for any $S$, with   effective constants $u_{\K,F}>0$ and $v_F>0$, and that the family $\big(F_n^{[s]}(\alpha)\big)_{1\le n \le v_F, s\ge 0}$ contains infinitely many irrational numbers. This theorem applies in particular when $F$ is an hypergeometric series with rational parameters or a multiple polylogarithm, and it encompasses a previous result by the second author and Marcovecchio in the case of polylogarithms. The proof relies on an explicit construction of Pad\'e-type approximants. It makes use of results of Andr\'e, Chudnovsky and Katz on $G$-operators, of a new linear independence criterion \`a la Nesterenko over number fields, of singularity analysis as well as of the saddle point method.
\end{abstract}

\section{Introduction}

The class of $G$-functions  was defined by Siegel~\cite{siegel} to generalize the Diophantine properties of the logarithmic function, by opposition 
to the exponential function which he generalized with the class of $E$-functions. 
A series $F(z)=\sum_{k=0}^\infty A_k z^k \in \Qbar[[z]]$ is a $G$-function if the following three conditions are met (we  fix an embedding of $\Qbar$ into $\mathbb C$):

1. There exists $C>0$ such that for any $\sigma \in \textup{Gal}(\Qbar/\mathbb Q)$ and any $k\ge 0$, $\vert \sigma(A_k)\vert \le C^{k+1}$.

2. Define $D_n$ as the smallest positive integer such that $D_nA_k$ is an algebraic integer for any $k\le n$. There exists 
$D>0$ such that for any $n\ge 0$, $D_n\le D^{n+1}$.

3. $F(z)$ is a solution of a linear differential equation with coefficients in $\Qbar(z)$.

\noindent The first property implies that the radius of convergence of $F$ is positive. In the second property, the existence of $D$ is enough for the 
purpose of this paper, but we mention that a famous conjecture of Bombieri implies that $D_n$ always divides $c^{n+1}d_{an}^b$ for some integers $a,b\ge 0, c\ge 1$, where 
$d_n:=\textup{lcm}\{1,2,\ldots,n\}=e^{n+o(n)}$ (see~\cite{firiK}). The third property shows that there is a number field containing all the coefficients $A_k$. In the case where they are all rational numbers, 
 the  three conditions become 
$\vert A_k\vert \le C^{k+1}$, $D_n A_k\in \mathbb Z$ for $k\le n$ and $D_n\le D^{n+1}$, and $F(z)$ is in fact a solution of a linear differential equation with coefficients 
in $\mathbb Q(z)$. 

 $G$-functions can be either algebraic over $\Qbar(z)$, like 
\begin{multline}\label{eq:400}
\sum_{k=0}^\infty z^k=\frac1{1-z}, \quad 
\sum_{k=0}^\infty \frac{\binom{2k}{k}}{k+1}z^k=\frac{2}{1+\sqrt{1-4z}}, 
\quad  \sum_{k=0}^\infty \binom{4k}{2k}z^k=\frac{\sqrt{1+\sqrt{1-6z}}}{\sqrt{2-12z}}, \\
\sum_{k=0}^\infty \binom{3k}{2k}z^k= \frac{2\cos\big(\frac13\arcsin(\frac32\sqrt{3z})\big)}{\sqrt{4-27z}}, \quad \sum_{k=0}^\infty \frac{(30k)!k!}{(15k)!(10k)!(6k)!}z^k, \qquad 
\end{multline}
or transcendental over $\Qbar(z)$, like  
\begin{multline}\label{eq:401}
\sum_{k=0}^\infty \frac{z^{k+1}}{k+1}=-\log(1-z), 
\quad \sum_{k=0}^\infty \frac{\binom{2k}{k}}{(k+1)^2}z^{k+1}=1-\sqrt{1-4z}+\log\Big(\frac{1+\sqrt{1-4z}}2\Big), 
\\
\sum_{k=0}^\infty \frac{\binom{2k}{k}}{2k+1}z^{2k+1}=2\arcsin(2z), \quad 
\sum_{k=0}^\infty\frac{z^{2k+2}}{(k+1)^2\binom{2k+2}{k+1}} = 2\arcsin\Big(\frac z2\Big)^2.\qquad 
\end{multline}
Transcendental  $G$-functions also include  the polylogarithms 
$\Li_s(z)=\sum_{k=1}^\infty \frac{z^k}{k^s}$ for   $s\geq 1$. All the above examples are special cases of 
the generalized hypergeometric series with rational parameters, which is a $G$-function: 
\begin{equation}\label{eq:403}
{}_{p+1}F_p \left[ 
\begin{matrix}
a_1, a_2, \ldots,  a_{p+1}
\\
b_1, b_2, \ldots, b_p
\end{matrix}
; z
\right] = \sum_{k=0}^\infty \frac{(a_1)_k(a_2)_k\cdots (a_{p+1})_k}{(1)_k(b_1)_k\cdots (b_p)_k} z^k, 
\end{equation}
where $(\alpha)_0=1$ and $(\alpha)_k=\alpha(\alpha+1)\cdots (\alpha+k-1)$ for $k\ge 1$; we assume that $-b_j \not\in\N=\{0,1,2,\ldots\}$ for any $j$.  Not all $G$-functions are hypergeometric, for instance the algebraic function  $\frac1{\sqrt{1-6z+z^2}}=\sum_{k=0}^\infty \big(\sum_{j=0}^k \binom{k}{j}
\binom{k+j}{j}\big)z^{k}$ or the transcendental functions $\sum_{k=0}^\infty \big(\sum_{j=0}^k \binom{k}{j}^2
\binom{k+j}{j}^2\big)z^{k}$, $\frac12\log(1-z)^2=\sum_{k=1}^\infty (\frac{1}{k}\sum_{j=1}^{k-1}\frac1j)z^k$, and more generally multiple polylogarithms 
$\sum_{n_1>\cdots>n_k\ge 1} \frac{z^{n_1}}{n_1^{s_1}n_2^{s_2}\cdots n_k^{s_k}}$ with $s_1, s_2, \ldots, s_k\in \mathbb Z.$

\medskip

In this paper, we are interested in the Diophantine properties of the values of $G$-functions at algebraic  points. We first recall that there is no definitive  theorem about the irrationality or transcendance of values of $G$-functions, like the Siegel-Shidlovsky Theorem for values of $E$-functions: transcendental  $G$-functions may take rational values or algebraic values at some non-zero algebraic points, see~\cite{arch, beuk93, wolf88} for examples related to Gauss ${}_2F_1$ hypergeometric function. 
Moreover, very few values of classical $G$-functions are known to be irrational: apart from logarithms of algebraic numbers (proved to be transcendental by other methods, namely the Hermite-Lindemann theorem), we may cite Ap\'ery's Theorem~\cite{apery} 
that $\zeta(3)=\Li_3(1)\notin \mathbb Q$, and  the Chudnovsky-Andr\'e Theorem~\cite{andrecrelle} on the algebraic independence over $\Qbar$ of the values ${}_{2}F_1[\frac 12,\frac 12;1;\alpha]$ and ${}_{2}F_1[-\frac12,\frac12;1;\alpha]$ for any $\alpha\in \Qbar$, $0<\vert \alpha\vert <1$~(\footnote{This result was first proved by G. Chudnovsky in the 70's by an indirect method not related to $G$-functions, and it was reproved by Andr\'e in the 90's by a method  designed for certain $G$-functions (simultaneous adelic uniformization), but which has been applied so far only to these ${}_{2}F_1$  functions.}). 

\medskip

Up to now, known results on values of $G$-functions can be divided into two families. The first one gathers theorems on $F(\alpha)$, where $\alpha\in\Qbar\subset\C$ is sufficiently close to 0 in terms of $F$ (and, often, of other parameters including the degree and height of $\alpha$).  
One of the most general results of this family is the following. 
\begin{theo}[Chudnovsky~\cite{chud1, chud2}]\label{theo:chud1} 
 Let $Y(z)={}^t(F_1(z), \ldots, F_S(z))$ be a vector of $G$-functions solution of 
a differential system $Y'(z)=A(z)Y(z)$, where $A(z)\in M_S(\Qb(z))$. Assume that $1, F_1(z)$, \ldots, $F_S(z)$ are $\Qbar(z)$-algebraically independent. 
Then for any integer $d\ge 1$, there exists $C=C(Y,d)>0$ such that, for any algebraic number $\alpha\neq 0$ of degree $d$ with  $\vert \alpha \vert <\exp(-C\log\left(H(\alpha)\right)^{\frac{4S}{4S+1}}),$ there does not exist a polynomial relation of degree $d$ and coefficients in $\mathbb Q(\alpha)$  between the  values $1, F_1(\alpha), \ldots, F_S(\alpha)$. 
\end{theo}
Here, $H(\alpha)$ is the naive height of $\alpha$, i.e. the maximum of the modulus of the integer coefficients of the (normalized) minimal polynomial of $\alpha$ over $\mathbb Q$. See~\cite{abd} for a general  strategy recently obtained to prove algebraic independence of $G$-functions. 
Chudnovsky's theorem refines the works 
of Bombieri~\cite{bombieri} and Galochkin~\cite{galoshkin}.   
Andr\'e~\cite{andrelivre} generalized Chudnovsky's theorem to the case of an inhomogenous system $Y'(z)=A(z)Y(z)+B(z)$. 
Thus, if we consider the case where $\alpha=a/b \in \mathbb  Q$ and  $d=1$,   the values $1, F_1(\alpha), \ldots, F_S(\alpha)$ are $\mathbb Q$-linearly independent provided $b\ge (c_1\vert a \vert)^{c_2}>0$, for some constants $c_1>0$ and $c_2>1$ depending on the vector $Y$. 
The best value known so far for $c_2$ is  quadratic in $S$; see~\cite{firi2, zud2} for related results. When $(1,F_1(z), \ldots, F_S(z))=(1,\Li_1(z), \ldots, \Li_S(z))$, we refer to~\cite{hataliouville, niki} for the best linear independence results, where $c_2$ is ``only'' linear in $S$.

\medskip

The second family consists in more recent results where $\alpha$ is a fixed algebraic point in the disk of convergence: lower bounds are obtained for the dimension of the vector space generated over a given number field by $F(\alpha)$, where $F$ ranges through a suitable set of $G$-functions. In general, this lower bound is not large enough to imply that all these values $F(\alpha)$ are irrational. In this family, we quote the theorem that 
 infinitely many odd zeta values $\zeta(2n+1)=\Li_{2n+1}(1)$, $n\ge 1$, are irrational  (see \cite{BR, rivoalcras}). Let us also quote the following result, first proved in~\cite{ribordeaux} when $\alpha$ is real.
 
\begin{theo}[Marcovecchio~\cite{marco2}]\label{theo:3} Let $\alpha\in \Qbar$, $0<\vert \alpha\vert <1$. The dimension of the $\mathbb Q(\alpha)$-vector space spanned by $1, \Li_1(\alpha), \ldots, \Li_S(\alpha)$ is larger than $\frac{1+o(1)}{[\mathbb Q(\alpha):\mathbb Q]\log(2e)}\log(S)$ as $S\to +\infty$.
\end{theo} 

It seems that all known results in this second family concern only specific $G$-functions, essentially polylogarithms. This is not the case of our main result,  Theorem \ref{theo:1} below, which is very general. 
 Starting from a $G$-function $F(z)=\sum_{k=0}^\infty A_k z^k$ with radius of convergence $R$, we define for any integers $n\ge 1$ and $s\ge 0$ the $G$-functions 
\begin{equation}\label{eq:405}
F_n^{[s]}(z)=\sum_{k=0}^\infty \frac{A_k}{(k+n)^s}z^{k+n}
\end{equation}
which all have $R$ as radius of convergence.

Let $\K$ be a number field that contains all the Taylor coefficients $A_k$ of $F$. 
For any integer $S$ and any $\alpha\in\K$ such that $0<|\alpha|<R$,  let $\Phi_{\alpha,S}$ denote the $\K$-vector space spanned by the numbers $F_{n}^{[s]}(\alpha)$ 
for $n\geq 1$ and $0\le s\le S$; of course $\Phi_{\alpha,S}$ depends also implicitly on  $F$ and $\K$. We shall obtain lower and upper bounds on 
$\dim_{\K} (\Phi_{\alpha,S})$ but to state them precisely, we need to introduce some notations.

We consider a differential operator $L=\sum_{j=0}^\mu P_j(z)(\frac{d}{dz})^j \in \Qbar[z,\frac{d}{dz}]$ such that $LF(z)=0$ and $L$ is of minimal order for $F$; then $L$ is a $G$-operator and in particular it is fuchsian by a  result of Chudnovsky~\cite{chud1, chud2}. We denote by $\delta$  the degree of $L$ and by $\omega\geq 0$ the multiplicity of 0 as a singularity of $L$, i.e. the order of vanishing of $P_\mu$ at $0$. We have $\delta = \deg(P_\mu)$ because $\infty$ is a regular singularity of $L$.
We let $\ell = \delta-\omega$, and $\lz = \max(\ell, \widehat f_1,\ldots,\widehat f_\eta)$ where 
$\widehat{f}_{1}$, \ldots, $\widehat{f}_{\eta}$ are the integer exponents of $L$ at  $\infty$ (so that $\lz = \ell$ if no exponent at $\infty$ is an integer). We refer to~\cite{hille} for the definitions and properties of these classical notions, and to \cite[\S 3]{andre} for those of $G$-operators.

\begin{theo} \label{theo:1} If $F$ is not a polynomial, then there exists an effective constant $C(F) > 0$ such that for any $\alpha\in \K$, $0< |\alpha| <R$, we have 
\begin{equation}\label{eq:200}
\frac{1+o(1)}{[\K:\Q]C(F)}  \log(S) \le \dim_{\K} (\Phi_{\alpha,S}) \le  \lz S+ \mu. 
\end{equation}
The second inequality holds for all $S\ge 0$ while in the first one, $o(1)$ is for $S\to +\infty$.
\end{theo}
The upper bound in  \eqref{eq:200} depends only on $F$. The constant $C(F)$ is independent from the number field $\K$, which is assumed to contain  $\alpha$ and all the Taylor coefficients $A_k$ of $F$; its expression involves certain quantities introduced in Proposition~\ref{prop:1} in \S\ref{subsec40}.

\medskip

We have the following corollary, in a case where $\lz=1$. The proof is given in~\S\ref{sec:examples}, together with many examples and other applications of Theorem \ref{theo:1}.
\begin{coro}\label{coro:1} Let us fix some rational numbers $a_1, \ldots, a_{p+1}$ and $b_1, \ldots, b_{p}$ such that
$a_i \not\in\Z\setminus\{1\}$ and $b_j\not\in-\N$ for any $i$, $j$. 
  Then for any 
 $\alpha\in\Qbar$ such that $0<|\alpha|<1$, 
infinitely many of the hypergeometric values 
\begin{equation}\label{eq:404}
\sum_{k=0}^\infty \frac{(a_1)_k(a_2)_k\cdots (a_{p+1})_k}{(1)_k(b_1)_k\cdots (b_p)_k} \frac{\alpha ^k}{(k+1)^s}, 
\quad s\ge 0
\end{equation}
are linearly independent over $\mathbb Q(\alpha)$.
\end{coro}
The numbers in \eqref{eq:404} are hypergeometric because they are equal to 
$$
{}_{p+s+1}F_{p+s} \left[ 
\begin{matrix}
a_1, a_2, \ldots,  a_{p+1}, 1,  \ldots, 1
\\
b_1, b_2, \ldots, b_p, 2, \ldots, 2
\end{matrix}
; \alpha
\right]
$$ 
where $1$ and $2$ are both repeated $s$ times. It seems to be the first general Diophantine result of this type for values of hypergeometric functions. Of course the conclusion of Corollary~\ref{coro:1} can be stated more precisely as 
$$\dim_{\Q(\alpha)}\Span_{\Q(\alpha)} \Big\{ \sum_{k=0}^\infty \frac{(a_1)_k(a_2)_k\cdots (a_{p+1})_k}{(1)_k(b_1)_k\cdots (b_p)_k} \frac{\alpha ^k}{(k+1)^s}, \; 0\leq s\leq S\Big\}\geq 
 \frac{1+o(1)}{[\Q(\alpha):\Q]C}\log(S), $$
 where $C>0$ depends on $a_1, \ldots, a_{p+1}$ and $b_1, \ldots, b_{p}$. The special case $p=0$, $a_1=1$ corresponds to Theorem \ref{theo:3} stated above, except that $\log(2e)$ is replaced with $C$. An {\em ad hoc} analysis in this special case would give $C=\log(2e)$, thereby providing Theorem \ref{theo:3} again (with a new proof, see below).

\medskip

The strategy to prove Theorem \ref{theo:1} is as follows. First, we construct 
certain algebraic   numbers  $\ppP_{j,t,s,n}\in \K$  and polynomials $\qqQ_{j,s,n}(z) \in \K[z]$  such that  for any $s,n\geq 1$:
\begin{equation}\label{eqBintro}
F_n^{[s]}(z)=\sum_{t=1}^{s}\sum_{j=1}^{\lz} \ppP_{j,t,s,n}F_j^{[t]}(z) + \sum_{j=0}^{\mu-1} \qqQ_{j,s,n}(z)\Big( z \frac{d}{d z}\Big)^j F (z),
\end{equation}
with geometric bounds on denominators and moduli of Galois conjugates (see Proposition~\ref{prop:1} in \S \ref{subsec40} for a precise statement). 
Eq.~\eqref{eqBintro} is a far reaching generalization of a property trivially satisfied by polylogarithms: for any $n\ge 1$, 
$$
\sum_{k=0}^\infty \frac{z^{k+n}}{(k+n)^{s}} = \Li_{s}(z) - 
\sum_{k=1}^{n-1} \frac{z^{k}}{k^{s}}.
$$
To obtain this result we study linear recurrences associated with $G$-operators, and make use in a crucial way of the results of Andr\'e, Chudnovsky and Katz~\cite{andre, firi}. With $z= \alpha$, \eqref{eqBintro} proves the inequality on the right-hand side of \eqref{eq:200}. This part of the proof of Theorem~\ref{theo:1} uses only methods with an algebraic flavor.

To prove the inequality on the left-hand side of \eqref{eq:200}, we use methods with a more Diophantine flavor. We consider the series
$$
T_{S,r,n}(z) = n!^{S-r}\sum_{k=0}^\infty \frac{k(k-1)\cdots (k-rn+1)}{(k+1)^S(k+2)^S\cdots (k+n+1)^S} \,A_k \,z^{-k}
$$
where $\vert z\vert >1/R$, $r$ and $n$ are integer  parameters such that $r\leq S $ and $n\to +\infty$. If $A_k=1$ for any $k$, this is essentially the series used in \cite{ribordeaux} and \cite{marco2} to prove Theorem \ref{theo:3}. Using \eqref{eqBintro} again, we prove that $T_{S,r,n}(1/\alpha ) $ is a $\K$-linear combination of the numbers $F_j^{[t]}(\alpha)$ ($1\leq t \leq S$,  $1 \leq j \leq \lz$) and $  ( z \frac{d}{d z})^j F (\alpha)$ ($0\leq j \leq \mu-1$). In fact, the series $T_{S,r,n}(z)$ can  be interpreted has an explicit Pad\'e-type approximant at $z=\infty$ for the functions $F_{j}^{[t]}(1/z)$ and $ ( z \frac{d}{d z})^j F (1/z)$.

We apply singularity analysis  and the saddle point method to prove that 
\begin{equation}\label{eqZintro}
T_{S,r,n}(1/\alpha ) = a^n n^\kappa \log(n)^\lambda \Big( \sum_{q=1}^Q c_q \zeta_q^n +o(1)\Big) \mbox{ as } n\to\infty,
\end{equation}
for some integers $Q\geq 1$ and $\lambda\geq 0$, real numbers $a>0$ and $\kappa$, non-zero complex numbers $c_1$,\ldots, $c_Q$ and pairwise distinct complex numbers $\zeta_1$, \ldots, $\zeta_Q$ such that $|\zeta_q|=1$ for any $q$. These parameters are effectively computed in terms of the finite singularities of $F$.

To conclude the proof we apply a linear independence criterion, as for all results of the second family mentioned above. Such a criterion enables one to deduce a lower bound on the dimension of the $\K$-vector space spanned by complex numbers $\vartheta_1$, \ldots,  $\vartheta_J$ from the existence of linear forms $T_n = \sum_{j=1}^J p_{j,n} \vartheta_j$ with coefficients $p_{j,n}\in\OK$. This lower bound is non-trivial if $|T_n|$
 is very small, and $p_{j,n}$ is not too large. However one more assumption is needed. In Siegel-type criteria this assumption is the non-vanishing of a determinant; Theorem \ref{theo:3} is proved in this way in \cite{marco2}, by constructing several sequences $(T_n^{(k)})$. On the opposite, Nesterenko's criterion~\cite{nest} (and its generalizations \cite{Topfer, Bedulev} to number fields) enables one to construct only one sequence $(T_n)$, but it requires a lower bound on $|T_n|^{1/n}$; this is how Theorem \ref{theo:3} is proved in~\cite{ribordeaux} if $\alpha$ is real. If $\liminf_n  |T_n|^{1/n}$ is smaller than $\limsup_n | T_n|^{1/n}$, this lower bound is weaker. In fact, in our situation, namely with the asymptotics \eqref{eqZintro}, it is not even clear that $\liminf_n |  T_n|^{1/n} $ is positive so that these criteria do not apply. 
 We solve this problem by generalizing Nesterenko's criterion (over any number field) to linear forms $(T_n)$ with asymptotics given by   \eqref{eqZintro}; our lower bound is best possible (see \S \ref{sec:critere} for precise statements). In the special case of polylogarithms, this provides a new proof of Theorem \ref{theo:3} when $\alpha$ is not real.
 
 \bigskip
 
 The structure of this paper is as follows. In \S \ref{sec:examples} we deduce Corollary \ref{coro:1} from  Theorem \ref{theo:1}, and give applications of these results. In \S \ref{sec:critere} we state and prove the generalization of  Nesterenko's linear independence criterion to linear forms with asymptotics given by   \eqref{eqZintro}. Then in \S \ref{seclinrec} we prove a general result, of independent interest, on linear recurrences related to  $G$-operators (using in a crucial way the  Andr\'e-Chudnovsky-Katz theorem). This result allows us to prove \eqref{eqBintro} in \S \ref{secpreuveprinc}, with geometric bounds on denominators and moduli of Galois conjugates. We conclude the proof of  Theorem \ref{theo:1} in \S  \ref{secpreuvethun}, except for the asymptotic estimate  \eqref{eqZintro} that we obtain in \S \ref{sec:asympT} using singularity analysis and the saddle point method. At last, we mention in \S \ref{secrempos} how to simplify the proof in  the special case where $A_k\geq 0$ for any $k$, and $\alpha>0$.

\section{Examples}\label{sec:examples}

 The generalized hypergeometric series defined by \eqref{eq:403}, if $b_j\not\in-\N$ for any $j$, 
 is solution of the differential equation $L_hy(z)=0$ where
$$
L_h=\theta(\theta+b_1-1)\cdots (\theta+b_p-1)-z(\theta+a_1)\cdots (\theta+a_{p+1}), \quad \theta=z\frac{d}{dz}.
$$
It is a $G$-function if and only if the $a_j$'s and $b_j$'s are rational numbers, in which case $L_h$ is a $G$-operator. Assuming $a_i\not\in-\N$, it is not a polynomial. We now compute the quantities defined before Theorem \ref{theo:1}, especially $\lz$. 
The degree $\delta$ of $L_h$ is $p+2$ and the multiplicity $\omega$ of $0$ as a singularity of $L_h$ is $p+1$. Hence, $\ell=\delta-\omega=1$ (consistently with the expression of $L_h$ and Lemma \ref{lemBV} below). Moreover, the exponents of $L_h$ at $0$ are $0, 1-b_1, \ldots, 1-b_p$, while those at $\infty$ are $a_1, \ldots, a_{p+1}$, so that  $
\lz = \max(1,  \widehat{a}_1, \ldots, \widehat{a}_{\eta})
$
where the $\widehat{a}_j$  are the integer parameters amongst $a_1$, \ldots, $a_{p+1}$. 
If none of the $a_j$'s is an  integer greater than 1  then $\lz=1$. This proves Corollary~\ref{coro:1}. 

We now list the hypergeometric parameters of the  examples stated in the Introduction:
$$
\frac{1}{k+1} \longleftrightarrow \left[ 
\begin{matrix} 1,  1\\2   
\end{matrix}
\right]
\quad  \quad 
\frac{\binom{2k}{k}}{k+1} \longleftrightarrow \left[ 
\begin{matrix} \frac12, 1\\ 2
\end{matrix}
\right]
\quad\quad   
\binom{3k}{2k} \longleftrightarrow \left[ 
\begin{matrix} \frac13, \frac23\\ \frac12
\end{matrix}
\right]
$$
$$
\binom{4k}{2k} \longleftrightarrow\left[ 
\begin{matrix} \frac14, \frac34\\ \frac12
\end{matrix}
\right]
\quad\quad
\frac{\binom{2k}{k}}{(k+1)^2} \longleftrightarrow \left[ 
\begin{matrix} \frac12, 1, 1\\ 2, 2
\end{matrix}
\right]
\quad\quad 
\frac1{(k+1)^2\binom{2k+2}{k+1}} \longleftrightarrow \left[ 
\begin{matrix} 1, 1, 1\\
\frac32, 2
\end{matrix}
\right]
$$
$$\frac{\binom{2k}{k}}{2k+1} \longleftrightarrow \left[ 
\begin{matrix} \frac12, \frac12 \\ \frac32
\end{matrix}
\right]\quad \quad 
\frac{(30k)!k!}{(15k)!(10k)!(6k)!} \longleftrightarrow \left[ 
\begin{matrix} \frac{1}{30}, \frac{7}{30}, \frac{11}{30}, \frac{13}{30}, \frac{17}{30}, \frac{19}{30}, \frac{23}{30}, \frac{29}{30}
\\ \frac{1}{5}, \frac{1}{3}, \frac{2}{5}, \frac{1}{2}, \frac{3}{5}, \frac{2}{3}, \frac{4}{5}
\end{matrix}
\right].
$$
In these eight cases, we have $\ell_0=1$ so that  Corollary~\ref{coro:1} applies (separately)  to them.

\medskip 

Let us now compute $\lz$ for four 
non-hypergeometric examples.  The function $\frac{1}{\sqrt{1-6z+z^2}}=\sum_{k=0}^\infty (\sum_{j=0}^k \binom{k}{j}\binom{k+j}{j})z^k$ is solution of the differential equation
$$
(z^2-6z+1)y'(z)+(z-3)y(z)=0
$$
which is minimal for this function; its exponent at 
$\infty $ is 
$1$. Hence $\lz = \ell =2$ and  Theorem~\ref{theo:1} provides  $ \frac{1+o(1)}{[\K:\Q]C}  \log(S)$ $\K$-linearly independent numbers amongst the 
  numbers 
$$
\sum_{k=0}^\infty \Big(\sum_{j=0}^k \binom{k}{j}\binom{k+j}{j}\Big) \frac{\alpha^k}{(k+1)^s} \quad \textup{and} \quad \sum_{k=0}^\infty \Big(\sum_{j=0}^k \binom{k}{j}\binom{k+j}{j}\Big) \frac{\alpha^k}{(k+2)^s}, \quad 0\leq s \leq S .
$$

The function $\frac12 \log(1-z)^2=\sum_{k=0}^\infty (\frac{1}{k+1}\sum_{j=1}^{k}\frac1j)z^{k+1}$ is solution of 
the differential equation
$$
(z-1)^2y'''(z)+3(z-1)y''(z)+y'(z)=0
$$
which is  minimal for this function;  its exponents 
at 
$\infty$ are $0,0,0$. Hence $\lz = \ell=2$  and  Theorem~\ref{theo:1} applies in the same way  to the numbers 
$$
\sum_{k=1}^\infty \Big(\sum_{j=1}^{k-1}\frac1j\Big) \frac{\alpha^k}{k^{s+1}} \quad \textup{and} \quad \sum_{k=1}^\infty \Big(\sum_{j=1}^{k-1}\frac1j\Big) \frac{\alpha^k}{k(k+1)^s}, \quad s\ge 0.
$$

The generating function of the Ap\'ery  numbers $\sum_{k=0}^\infty (\sum_{j=0}^k \binom{k}{j}^2\binom{k+j}{j}^2)z^k$ is solution of 
the minimal differential equation
$$
z^2(1-34z+z^2)y'''(z) +z(3-153z+6z^2)y''(z)+(1-112z+7z^2)y'(z) + (z-5)y(z)=0.
$$
Its exponents at 
$\infty$ are $1,1,1$. Hence $\lz = \ell=2$  and  Theorem~\ref{theo:1} applies again to the numbers 
$$
\sum_{k=0}^\infty \Big(\sum_{j=0}^k \binom{k}{j}^2\binom{k+j}{j}^2\Big) \frac{\alpha^k}{(k+1)^s} \quad \textup{and} \quad \sum_{k=0}^\infty \Big(\sum_{j=0}^k \binom{k}{j}^2\binom{k+j}{j}^2\Big) \frac{\alpha^k}{(k+2)^s}, \quad s\ge 0.
$$

We conclude this section with the case of the series $G_b(z)=\sum_{k=1}^\infty \frac{\chi(k)}{k^b}z^{k}$ where $b$ is any fixed positive integer 
and $\chi$ is the unique non-principal character mod 4. Since $G_b(z)=\sum_{k=0}^\infty \frac{(-1)^{k}}{(2k+1)^b}z^{2k+1}$, it is a $G$-function. 
Moreover, 
$
\theta \big((1+z^2)\theta^b\big)G_b(z)=0$, which is of minimal order for $G_b(z)$. Hence $\theta \big((1+z^2)\theta^b\big)$ is a $G$-operator: it is such that  
$\mu=b+1$, $\delta=b+3$, $\omega=b+2$, $\ell=1$ and its exponents at infinity 
are $0,0, \ldots, 0,2$, where $0$ is repeated $b$ times. Hence $\lz=2$ and  Theorem~\ref{theo:1} applies to the numbers 
$$
\sum_{k=1}^\infty \frac{\chi(k)}{k^{b+s}}\alpha^k \quad \textup{and} \quad \sum_{k=1}^\infty \frac{\chi(k)}{k^{b}(k+1)^s}\alpha^k,  \quad s\ge 0.
$$
More generally, Theorem~\ref{theo:1} applies to any $G$-function of the form $\sum_{k=1}^\infty \frac{\chi(k)}{A(k)}z^{k}$ where $\chi$ is a  Dirichlet character and $A(X)\in \mathbb Q[X]$ is split over $\mathbb Q$ and such that $A(k)\neq 0$ for any positive integer $k$.

\section{Generalization of  Nesterenko's linear independence criterion}\label{sec:critere}

The following version of Nesterenko's linear independence criterion will be used in the proof of Theorem \ref{theo:1}.

\medskip

Let $\K$ be a number field embedded in $\C$. We let $\Lcor = \R$ if $\K\subset\R$, and $\Lcor = \C$ otherwise. 
   We denote   by $o(1)$ any sequence that tends to 0 as $n\to\infty$.

\begin{theo}\label{th:nest} 
Let $(Q_n)$ be an increasing sequence of positive real numbers, with limit $+\infty$, such that $Q_{n+1} = Q_n^{1+o(1)}$. Let $T\geq 1$, $c_1$,\ldots, $c_T$ be non-zero complex numbers, and $\zeta_1$, \ldots, $\zeta_T$ be pairwise distinct complex numbers such that $|\zeta_t| = 1$ for any $t$.

Consider $N$  numbers $\vartheta_1, \ldots, \vartheta_N\in\Lcor$. Assume that for some $\tau > 0$ there exist $N$ sequences $(p_{j,n})_{n\ge 0}$, $j=1, \ldots, N$, such that for any $j$ and $n$,    $p_{j,n}\in \OK$,  all Galois conjugates of $p_{j,n}$ have modulus less than $Q_n^{1+o(1)} $, 
and  
\begin{equation} \label{eqhypnest}
  \sum_{j=1}^N p_{j,n} \vartheta_j    =  Q_n^{-\tau+o(1)} \Big( \sum_{t=1}^T c_t\zeta_t^n+o(1)\Big).
  \end{equation}
Then 
$$
\dim_\K \textup{Span}_{\K} (\vartheta_1, \ldots, \vartheta_N)    \geq \frac{\tau+1}{[\K:\Q]}.
$$
\end{theo}

Given $0 < \alpha < 1 < \beta$ and $\kappa\in\C$, $\lambda\in\C$, this theorem can be applied when
 all Galois conjugates of $p_{j,n}$ have modulus less than $\beta^{n(1+o(1))}$  and 
\begin{equation} \label{eqA36}
  \sum_{j=1}^N p_{j,n} \vartheta_j    =  \alpha^n n^\kappa (\log n)^\lambda \Big(  \sum_{t=1}^T c_t\zeta_t^n+o(1)\Big);
    \end{equation}
then the conclusion reads
$$
\dim_\K \textup{Span}_{\K} (\vartheta_1, \ldots, \vartheta_N)\ge  \frac{1}{[\K:\Q]} \Big( 1 - \frac{\log(\alpha)}{\log(\beta)}\Big).
$$

\medskip

Nesterenko's original linear independence criterion~\cite{nest} is a general quantitative result, of which Theorem \ref{th:nest} is  a special case if $\K=\Q$, $T=1$, $\zeta_1=\pm 1$. The case where $\K=\Q$, $T=2$, $\zeta_2 = \overline{\zeta_1}$ and $c_2 = \overline{c_1}$ follows using either lower bounds for linear forms in logarithms (if $c_1,\zeta_1\in\Qbar$, see  \cite{SorokinZR} or \cite[\S 2.2]{SFoscillate}) or Kronecker-Weyl's equidistribution theorem \cite{SFoscillate}. 

Nesterenko's   criterion has been extended to any number field $\K$ by T\"opfer \cite{Topfer} and Bedulev \cite{Bedulev}; their results are similar, but different in several aspects. The case $T=1$ of Theorem \ref{th:nest} follows from 
T\"opfer's Korollar 2 \cite{Topfer}, but does not seem to follow directly from Bedulev's result since he uses the exponential Weil height relative to $\K$ instead of the house of $p_{j,n}$ (i.e., the maximum of the moduli of all Galois conjugates of $p_{j,n}$). 

  We shall deduce the general case of Theorem \ref{th:nest} from T\"opfer's result using Vandermonde determinants (as in the proof of \cite[Lemma 6]{firi}). This provides also a new and simpler proof of the above-mentioned case  $\K=\Q$, $T=2$, $\zeta_2 = \overline{\zeta_1}$ and $c_2 = 
\overline{c_1}$.

\medskip

Even in the special case where $T=1$ and $\K=\Q$, the lower bound in Theorem \ref{th:nest}  is best possible (see \cite{FHKL}). 
We have the following corollary, which we shall not use in this paper but which can be useful in other contexts.
\begin{coro} \label{coronest}  Let $\alpha,\beta\in\R$ be such that $0 < \alpha < 1 < \beta$. Consider $N$  numbers $\vartheta_1, \ldots, \vartheta_N\in\Lcor$. Assume that there exist $N$ sequences $(p_{j,n})_{n\ge 0}$, $j=1, \ldots, N$, such that for any $j$ and $n$,    $p_{j,n}\in \OK$,  all Galois conjugates of $p_{j,n}$ have modulus less than $\beta^{n(1+o(1))}$, and 
$$\limsup_{n\to\infty} \Big|   \sum_{j=1}^N p_{j,n} \vartheta_j    \Big| ^{1/n} \leq \alpha.$$
Assume also that $  \sum_{j=1}^N p_{j,n} \vartheta_j \neq 0$ for infinitely many $n$, and that 
for any $j$ the function $\sum_{n=0}^{\infty} p_{j,n}z^n$ is solution of a homogeneous linear differential equation with coefficients in $\Qbar(z)$. Then 
$$
\dim_\K \textup{Span}_{\K} (\vartheta_1, \ldots, \vartheta_N)\ge  \frac{1}{[\K:\Q]} \Big( 1 - \frac{\log(\alpha)}{\log(\beta)}\Big).
$$
\end{coro}
The point in Corollary \ref{coronest} is that no lower bound is needed on $|   \sum_{j=1}^N p_{j,n} \vartheta_j  |$. This result fits in the context of $G$-functions, since its assumptions imply that $\sum_{n=0}^{\infty} p_{j,n}z^n$ is a $G$-function for any $j$. 
To deduce Corollary \ref{coronest} from Theorem \ref{th:nest}, it is enough to notice that $ \sum_{n=0}^{\infty}  \sum_{j=1}^N p_{j,n} \vartheta_j z^n =  \sum_{j=1}^N  \vartheta_j 
\sum_{n=0}^{\infty} p_{j,n}z^n $ is  solution of a homogeneous linear differential equation with coefficients in $\Qbar(z)$. We can then apply classical transfer results from Singularity Analysis: an asymptotic estimate like \eqref{eqA36} holds.

\medskip

\begin{proof}[Proof  of Theorem \ref{th:nest}]: For any $n\geq 0$ we consider the following determinant:
$$\Delta_n = \left\vert \begin{matrix} \zeta_1^n & \ldots & \zeta_T^n \\ \vdots && \vdots \\ \zeta_1^{n+T-1} & \ldots & \zeta_T^{n+T-1}\end{matrix}\right\vert.$$
We have $| \Delta_n|  = | \zeta_1^n \ldots\zeta_T^n\Delta_0| = |\Delta_0| \neq 0$ since $\Delta_0$ is the Vandermonde determinant built on the pairwise distinct complex numbers $\zeta_1$, \ldots, $\zeta_T$. We claim that for any $n\geq 0$ there exists $\delta_n \in \{0,\ldots,T-1\}$ such that
\begin{equation} \label{eqjn}
\Big|  \sum_{t=1}^T c_t\zeta_t^{n+\delta_n}\Big| \geq \frac{|c_1 \Delta_0|}{T!}.
  \end{equation}
Indeed if this equation holds for no integer $\delta_n \in\{0,\ldots,T-1\}$ then upon replacing $C_{1,n}$ with $\frac1{c_1}\sum_{t=1}^T c_t C_{t,n}$ (where $C_{t,n}$ is the $t$-th column of the matrix of which $\Delta_n$ is the determinant) we obtain:
$$ |\Delta_0| = |\Delta_n|   < \frac{T}{c_1}    \frac{|c_1 \Delta_0|}{T!} (T-1)! =  |\Delta_0| ,$$
since all minors of size $T-1$ have modulus less than or equal to $(T-1)!$. This contradiction proves the claim \eqref{eqjn} for some $\delta_n \in \{0,\ldots,T-1\}$.

Now let $p'_{j,n} = p_{j, n+\delta_n}$. Since $Q_{n+1} = Q_n^{1+o(1)}$ and $0 \leq \delta_n \leq T-1$ (where $T-1$ does not depend on $n$), 
 all Galois conjugates of $p_{j,n}$ have modulus less than $ Q_n^{1+o(1)}$. Moreover \eqref{eqjn} yields $  \vert \sum_{j=1}^N p'_{j,n} \vartheta_j \vert  =  Q_n^{-\tau+o(1)} $. Therefore
 T\"opfer's Korollar 2  \cite{Topfer} applies to the sequences $(p'_{j,n})$: this concludes the proof of Theorem \ref{th:nest}.
\end{proof}

\begin{Remark} In the proof of Theorem \ref{th:nest} the sequences $(p'_{j,n})$ may be such that $p'_{j,n} = p'_{j,n'}$ for some $n<n'$ even if this does not happen with $p_{j,n}$. This is not a problem since in this case $n'-n\leq T-1$, where $T$ is independent from $n$.
\end{Remark}

\section{Linear recurrences associated with $G$-operators} \label{seclinrec}

In this section we  apply some results of  
Andr\'e, Chudnovsky and Katz to prove a few general properties of $G$-operators (stated in \S \ref{subseclinrec}). We recall that for any $G$-function $F$, any 
 differential operator $L \in \Qbar[z,\frac{d}{dz}]$ of minimal order  such that $LF=0$   is a $G$-operator. We refer to \cite[\S 3]{andre} for   the definition and properties of $G$-operators.

\subsection{Setting and statements} \label{subseclinrec}

\begin{lem} \label{lemBV}
Let $\K$ be a number field, and $L = \sum_{j=0}^\mu P_j(z)\left(\frac d{dz}\right)^j$   a $G$-operator  with $P_j \in\K[X]$ and $P_\mu\neq 0$; denote by $\delta$ the degree of $L$, and by $\omega\geq 0$ the multiplicity of 0 as a singularity of $L$; let $\ell = \delta-\omega$. 

Then there exist some polynomials $Q_j(X)\in \mathbb \OK[X]$   and a positive rational  integer $\alpha$  such that
$$\alpha z^{\mu-\omega} L = \sum_{j=0}^\ell z^{j}Q_j(\theta +j) \mbox{ where } \theta =z\frac{d}{dz}.$$
Moreover letting $d_j = \deg (Q_j)$ we have
$$d_j \leq \mu \mbox{ for any $ 0 \leq j \leq \ell$, and } d_0 = d_\ell = \mu.$$
At last, $Q_0(X) = 0$ and $Q_\ell (-X+\ell)=0$ are (up to a multiplicative constant) the indicial equations of $L$  at 0 and $\infty$, respectively.
\end{lem}

This lemma belongs to folklore (see for instance \cite[\S 4.1]{BV} for a part of it) but for the sake of completeness we provide a proof in \S \ref{subsec22} below.

In what follows we keep the notation and assumptions of Lemma \ref{lemBV}.  We   denote by $\widehat e_{1}$, \ldots, $\widehat e_\kappa$  and  $\widehat f_{1}$, \ldots, $\widehat f_\eta$ the integer  exponents of $L$ at 0 and $\infty$, respectively; they are the integer  roots of the indicial equations at 0 and $\infty$, namely $Q_0(X) = 0$ and $Q_\ell (-X+\ell)=0$. We let $m\geq 1$ be such that 
$$
m > - \widehat{e}_{i} \mbox{ and } m > \widehat{f}_{j} - \ell  \mbox{ for all } 1\le i\le \kappa, 1\leq j \leq \eta ; 
$$
of course the condition on $ \widehat{e}_{i}$ (resp. $ \widehat{f}_{j} $) is  always satisfied if $\kappa=0$ (resp. $\eta=0$).
 Then  $Q_0(-n)\neq 0$ and $Q_\ell(-n)\neq 0$   for any integer $n\geq m$, so that 
the  linear recurrence relation 
\begin{equation}\label{eq:12b}
\sum_{j=0}^\ell Q_j(-n) U(n+j) =0,\quad n\ge m
\end{equation}
(satisfied by the Taylor coefficients of any power series in $1/z$ annihilated by $L$, see Step~1 in the proof of Lemma  \ref{lemlinrec}) 
has a $\mathbb C$-basis  of solutions 
 $(u_{1}(n))_{n\ge m}, \ldots, (u_{\ell}(n))_{n\ge m}$  with $u_j(n)\in\K$ for any $1\leq j \leq \ell$ and any $n\geq m$.  
The determinant  
\begin{equation}\label{eq:8b}
W(n)=\left \vert \begin{matrix}
u_1(n+\ell-1)&\cdots & u_{\ell}(n+\ell-1)
\\
u_1(n+\ell-2)&\cdots &u_{\ell}(n+\ell-2)
\\
\vdots & \vdots & \vdots 
\\
u_1(n)&\cdots &u_\ell(n)
\end{matrix}\right\vert
\end{equation}
is called a wronskian (or casoratian) of the recurrence.  

\medskip

Now let us consider an inhomogeneous linear recurrence relation
\begin{equation}\label{eq:4}
\sum_{j=0}^\ell Q_j(-n) V(n+j) = g(n),\quad n\ge m
\end{equation}
where $g(n)$ is defined for any $n\ge m$.
We let  $\Delta_j(n)= D_j(n)\frac{g(n-1)}{Q_\ell(1-n)}$ for $n\geq m+1$, where 
\begin{equation}\label{eqdjn}
D_j(n)=(-1)^j\left \vert \begin{matrix}
u_1(n+\ell-2)&\cdots &u_{j-1}(n+\ell-2)&u_{j+1}(n+\ell-2) &\cdots & u_{\ell}(n+\ell-2)
\\
\vdots & \vdots & \vdots & \vdots &\vdots&\vdots 
\\
u_1(n)&\cdots &u_{j-1}(n)&u_{j+1}(n)&\cdots & u_{\ell}(n)
\end{matrix}\right\vert .
\end{equation}

\begin{lem}\label{lemlinrec} The general solution of the recurrence \eqref{eq:4} is 
$$
V(n) = \sum_{j=1}^\ell \Big(\ddj +\sum_{k=m+1}^{n} \frac{\Delta_{j}(k)}{W(k)} \Big) u_{j}(n), \quad n\ge m , 
$$
where   $\ddjj_1$, \ldots, $\ddjj_\ell$  are arbitrary complex numbers. Moreover we have $W(n) \neq 0$ for any $n\geq m$, and the power series $\sum_{n = m}^{\infty} \frac{z^n}{W(n)}$,  $\sum_{n=m}^\infty u_j(n)z^n$ and  $\sum_{n=m+1}^\infty \frac{D_j(n)}{Q_\ell(1-n)}z^n$ (with $1\leq j \leq \ell$) are  $G$-functions.
\end{lem}

The first part of this lemma (namely, the expression of $V(n)$) is valid as soon as $Q_0(-n)Q_\ell(-n)\neq 0$ for any $n\geq m$: it does not rely on the assumption that $L$ is a $G$-operator.

\subsection{Proofs} \label{subsec22}

\begin{proof}[Proof of Lemma \ref{lemBV}]
Since $\infty$ is a regular singularity of $L$ we have $\deg (P_k) \leq \delta-(\mu-k)$ for any $k$ and $\deg (P_\mu) = \delta$; since $0$  is a regular singularity the order of vanishing of each $P_k$ at 0 is at least $\max(0, \omega- (\mu-k))$. Therefore we may 
 write 
 $$P_k(z) = \sum_{i= \max(0, \omega-  \mu+k )}^{\delta-\mu+k} p_{k,i} z^i.$$
  Now observe that for any $n\geq 1$, 
$$z^n \Big( \frac{d}{dz} \Big)^n = \sum_{j=1}^n c_{j,n} \theta^j \mbox{ with } c_{j,n}\in\Q \mbox{ and } c_{n,n}=1.$$
Then we let $p_{0,i}=0$ if $i\leq-1$, and 
$$S_k(X) = p_{0,k-\mu}+ \sum_{j=1}^\mu \Big( \sum_{i= \max(0, k+j-\mu)}^k p_{i+\mu-k, i}c_{j, i+\mu-k}\Big) X^j$$
for $\omega \leq k \leq \delta$, so that $\deg (S_\omega) =  \deg (S_\delta)= \mu$ and 
$$L = \sum_{k=\omega}^\delta z^{k-\mu} S_k(\theta).$$
Let  $\alpha \geq 1$ denote a common denominator of the (algebraic) coefficients of all polynomials $S_k$, and put $Q_j(X) = \alpha S_{j+\omega}(X-j)$. Then we have 
$$\alpha z^{\mu-\omega} L = \sum_{j=0}^\ell z^{j}Q_j(\theta +j).  $$
At last, since   $  \theta  z^p = p  z^p$ for any $p\in\Z$ we have
$$L z^p = \frac1{\alpha} \sum_{j=0}^\ell Q_j(p+j)z^{p+j+\omega-\mu}.$$
Since $Q_0$ and $Q_\ell$ have degree $\mu$, they are non-zero and (up to the multiplicative constant $ \frac1{\alpha} $) the  indicial equations at 0 and infinity are  respectively $Q_0(p) = 0$ and $Q_\ell (-p+\ell)=0$.
\end{proof}

\smallskip

\begin{proof}[Proof of Lemma \ref{lemlinrec}] We split the proof into four steps. Step 3 and a part of Step 2 are somewhat classical, and do not rely on the assumption that $L$ is a $G$-operator (see for instance \cite[pp. 5 and 22]{norlund}). However we provide a complete proof for the reader's convenience.

\medskip

\noindent {\bf Step 1:} Proof that $\sum_{n=m}^\infty u_{j}(n) z^n$ is a $G$-function.

For any power series $U(z) = \sum_{n=m}^\infty u_nz^{-n}$, Lemma \ref{lemBV} yields
$$\alpha z^{\mu-\omega} L U(z) = \sum_{k = m-\ell} ^{+\infty} \Big( \sum_{j=\max(0, m-k)}^{\ell} u_{k+j} Q_j(-k)\Big) z^{-k}.$$
Therefore $(u_n)_{n\geq m}$ is a solution of \eqref{eq:12b} if, and only if, $\alpha z^{\mu-\omega} L U(z)  = z^{1-m}U_0(z)$ where $U_0$ is a polynomial of degree at most $\ell-1$. 
In this case $(\frac{d}{d z})^{\ell} z^{\mu-\omega+m-1}L$ is a $G$-operator that annihilates $U(z)$; notice that $z^{\mu-\omega }L\in \K[z, \frac{d}{d z}]$ even if $\omega>\mu$, since 0 is a regular singularity of $L$. Applying the 
 Andr\'e-Chudnovsky-Katz theorem (see \cite[p. 719]{andre} or \cite[\S 4.1]{firi}), we deduce that if $u_n\in\K$ for any $n$ then $U(z)$ is a $G$-function in $1/z$.

\medskip

\noindent {\bf Step 2:} Computation of the wronskian $W(n)$.

First of all, if $W(n_0)=0$ for some $n_0\geq m$ then we obtain $\lambda_1,\ldots,\lambda_\ell\in\C$, not all zero, such that $\lambda_1 u_1(n) + \ldots + \lambda_\ell u_\ell(n) =0$ for any $n_0 \leq n \leq n_0 +\ell-1$; by induction this equality holds for any $n\geq m$, which is a contradiction. Therefore we have $W(n)\neq 0$ for any $n\geq m$. Moreover 
   $W(n)$ is solution of the linear recurrence of order $1$ 
\begin{equation}\label{eq:13}
Q_\ell(-n) W(n+1)=(-1)^{\ell}Q_0(-n)W(n),
\end{equation}
since the left hand side is equal to 
$$
\left \vert \begin{matrix}
-\sum_{j=0}^{\ell-1} Q_j(-n) u_1(n+j) & \ldots  &-\sum_{j=0}^{\ell-1} Q_j(-n) u_\ell(n+j)  \\ 
u_1(n+\ell-1)  &\ldots & u_\ell(n+\ell-1) \\
\vdots & & \vdots \\ 
u_1(n+1)  &\ldots & u_\ell(n+1)
\end{matrix}\right\vert.
$$
By Lemma \ref{lemBV} we have
$$
Q_0(X)=\gamma_0  \prod_{i=1}^\mu  (X-e_{i})  \quad  \mbox{ and }  \quad 
Q_\ell(X)=\gamma_\ell  \prod_{i=1}^\mu  (X+ f_{i} - \ell ),
$$
where $\gamma_0 , \gamma_\ell $ are non-zero elements of $\K$ and $e_1$, \ldots, $e_\mu $ (resp. $f_1$, \ldots, $f_\mu$) are the exponents of $L$ at 0 (resp. at $\infty$). 
  By Katz' theorem~\cite[p. 719]{andre},  these   exponents    are rational numbers. The recurrence \eqref{eq:13} is easily solved: for $n\ge m$, we have 
\begin{align*}
W(n)&= (-1)^{\ell(n-m)}\frac{Q_0(1-n)\cdots Q_0(-m)}{Q_\ell(1-n)\cdots Q_\ell(-m)}W(m) 
\\
&=W(m)\big((-1)^{\ell} \gamma_0  /  \gamma_\ell   \big)^{n-m} \prod_{i=1}^\mu  \frac{(n-1  + e_{i}  )\cdots (m+ e_{i}  )}{(n-1-f_{i} + \ell )\cdots (m-f_{i} + \ell)}    
\\
&= W(m)\big((-1)^{\ell} \gamma_0  /  \gamma_\ell   \big)^{n-m} \prod_{i=1}^\mu  \frac{(m+ e_{i} )_{n-m}}{(m-f_{i} + \ell)_{n-m}}. 
\end{align*}
Therefore  $\sum_{n = m}^{\infty} \frac{z^n}{W(n)}$ is a  ${}_{d+1} F_d$ hypergeometric series with rational parameters,  and accordingly  a $G$-function.

\medskip

\noindent {\bf Step 3:} Computation of the  solutions of \eqref{eq:4}.

Since $W(n)\neq 0$ for any $n\geq m$, given any sequence $(v(n))_{n\geq m}$ there exist sequences $(c_j(n))_{n\geq m}$, $1\leq j \leq \ell$, such that 
\begin{equation} \label{eq:7av}
v(n+k)=\sum_{j=1}^\ell c_j(n)u_j(n+k), \quad k=0, \ldots, \ell-1.
\end{equation}
This equation with $n+1$ and $k-1$ reads
$$
v(n+k)=\sum_{j=1}^\ell c_j(n+1)u_j(n+k), \quad k=1, \ldots, \ell-1, 
$$
so that 
\begin{equation}\label{eq:7}
\sum_{j=1}^\ell (\Delta c_j(n)) u_j(n+k) = 0, \quad k=1, \ldots, \ell-1, 
\end{equation}
where we define as usual the difference operator $\Delta x_n := x_{n+1}-x_n$. 

 Now let us assume that
$(v(n))_{n\geq m}$ is a solution of the inhomogeneous linear recurrence relation~\eqref{eq:4}. 
Since \eqref{eq:7av} with $n+1$ and $\ell-1$ yields
$$
v(n+\ell) = \sum_{j=1}^\ell (\Delta c_j(n)) u_j(n+\ell) + \sum_{j=1}^\ell c_j(n)u_j(n+\ell),  
$$
we obtain using also \eqref{eq:7av} and the fact that  $(u_j(n))_{n\geq m}$ is a solution of \eqref{eq:12b} for any $j$:
\begin{equation}\label{eq:8}
\sum_{j=1}^\ell (\Delta c_j(n)) u_j(n+\ell) = \frac{g(n)}{Q_\ell(-n)}.
\end{equation}
The $\ell$ equations given by \eqref{eq:7} and \eqref{eq:8} form a system of linear equations which enables us to find $\Delta c_j(n)$ by Cram\'er's rule because the determinant of the system is the wronskian
$W(n+1)$ defined by~\eqref{eq:8b}. We have $W(n+1)\neq  0$ (by Step 2) so that 
$
\Delta c_j(n) = \frac{\Delta_j(n+1)}{W(n+1)}
$
since  $\Delta_{j}(n) =  D_j(n)\frac{g(n-1)}{Q_\ell(1-n)}$ is equal to  the following determinant:
$$
\left \vert \begin{matrix}
u_1(n+\ell-1)&\cdots &u_{j-1}(n+\ell-1)&\frac{g(n-1)}{Q_\ell(1-n)}&u_{j+1}(n+\ell-1) &\cdots & u_{\ell}(n+\ell-1)
\\
u_1(n+\ell-2)&\cdots &u_{j-1}(n+\ell-2)&0&u_{j+1}(n+\ell-2) &\cdots& u_{\ell}(n+\ell-2)
\\
\vdots & \vdots & \vdots & \vdots &\vdots&\vdots &\vdots 
\\
u_1(n)&\cdots &u_{j-1}(n)&0&u_{j+1}(n) &\cdots& u_{\ell}(n)
\end{matrix}\right\vert.
$$
Therefore we obtain
$$
c_j(n)=c_j(m)+\sum_{k=m}^{n-1}\Delta c_j(k) = c_j(m)+\sum_{k=m+1}^{n} \frac{\Delta_j(k)}{W(k)}, \quad n\ge m,
$$
and finally
$$
v(n)=\sum_{j=1}^\ell \Big(c_j(m)+\sum_{k=m+1}^{n} \frac{\Delta_j(k)}{W(k)}\Big)u_j(n).
$$
Conversely, the same computations prove that any sequence defined in this way (with arbitrary constants $c_j(m)$, $1\leq j \leq \ell$) is 
a solution of the inhomogeneous linear recurrence relation~\eqref{eq:4}. 

\medskip

\noindent {\bf Step 4:} Proof that $\sum_{n=m+1}^\infty \frac{D_j(n)}{Q_\ell(1-n)}z^n$ is a $G$-function.

Expanding the determinant in \eqref{eqdjn} we see that $\sum_{n=m+1}^\infty  D_j(n) z^n$ is a $\Z$-linear combination of Hadamard (i.e., coefficientwise) products of $G$-functions $\sum_{n=m+1}^\infty  u_h(n+i) z^n$, so that it is a $G$-function. On the other hand $\sum_{n=m+1}^\infty \frac{z^n}{Q_\ell(1-n)}$  is  a $G$-function because $Q_\ell $ is split over the rationals (see Step 2)  so that finally $\sum_{n=m+1}^\infty \frac{D_j(n)}{Q_\ell(1-n)}z^n$ is a $G$-function for any $j$.
\end{proof}

\section{Properties of $F_n^{[s]}(z)$} \label{secpreuveprinc}

Throughout this section, let  $\K$ be a number field, and $L = \sum_{j=0}^\mu P_j(z)\left(\frac d{dz}\right)^j$   a $G$-operator  with $P_j \in\K[X]$ and $P_\mu\neq 0$. We denote by $\delta$ the degree of $L$ (i.e., $\delta = \deg (P_\mu)$ since $\infty$ is a regular singularity of $L$),  by $\omega\geq 0$ the multiplicity of 0 as a singularity of $L$ (i.e., the order of vanishing of $P_\mu$ at $0$), and we let $\ell = \delta-\omega$.  

Let $F(z) = \sum_{k=0}^\infty A_k z^k$, with $A_k\in\K$, be such that $LF=0$; then $F$ is a $G$-function. Of course, starting with such a $G$-function $F$, one may choose for $L$ a differential operator, of minimal order, such that $LF=0$; then $L$ is a $G$-operator.

Recall that we let $\theta =  z \frac{d}{dz}$ and that
$$F_n^{[s]}(z)=\sum_{k=0}^\infty \frac{A_k}{(k+n)^s}z^{k+n}.
$$

\subsection{The main proposition} \label{subsec40}

As in the introduction,  we denote by $\widehat{f}_{1}$, \ldots, $\widehat{f}_{\eta}$ the integer exponents of $L$ at  $\infty$, with $\eta=0$ if there isn't any.

\begin{prop}  \label{prop:1} Let  $m\ge 1$ be such that 
\begin{equation} \label{eqhypm}
  m > \widehat{f}_{j} - \ell  \; \mbox{ for all }    1\le j \leq \eta.
\end{equation}
Then for any $s,n\geq 1$:

$(i)$ There exist some algebraic   numbers  $\ppP_{j,t,s,n}\in \K$, and some polynomials $\qqQ_{j,s,n}(z) \in \K[z]$ of degree at most $n+s(\ell -1)$,  such that 
\begin{equation}\label{eq:406}
F_n^{[s]}(z)=\sum_{t=1}^{s}\sum_{j=1}^{\ell+m-1} \ppP_{j,t,s,n}F_j^{[t]}(z) + \sum_{j=0}^{\mu-1} \qqQ_{j,s,n}(z)(\theta^j F)(z).
\end{equation}

$(ii)$ All Galois conjugates of all the numbers   $  \ppP_{j,t,s,n} $  ($j\le\ell +m-1$, $t\le s$), and   all Galois conjugates of all the coefficients  of the polynomials $\qqQ_{j,s,n}(z)$  ($j\le \mu-1$), have modulus less than  $H(F,s,n)>0$ with  
$$
\limsup_{n\to + \infty } H(F,s,n)^{1/n} \le  C_1(F)^s
$$ 
for some constant $C_1(F)\ge 1$ independent of $s$.

$(iii)$ Let $D(F,s,n)>0$ denote the least common denominator of the algebraic numbers $\ppP_{j,t,s,n'}$ ($j\le\ell + m-1$, $t\le s$, $n'\leq n$) and of the coefficients of the polynomials $\qqQ_{j,s,n'}(z)$ ($j\le \mu-1$, $n'\le n$); then
$$
\limsup_{n\to + \infty } D(F,s,n)^{1/n} \le  C_2(F)^s
$$ 
for some constant $C_2(F)\ge 1$ independent of $s$.
\end{prop}
The constants $C_1(F)$ and $C_2(F)$ are effective and could be computed in principle. However, their values are complicated to write down and do not add much value. 

Of course the most interesting case of Proposition \ref{prop:1} is when $m = \max(1, \widehat{f}_{1}+1-\ell, \ldots, \widehat{f}_{\eta}+1-\ell)$, that is $m = \ell_0 - \ell + 1$ where $\ell_0$ was  defined in the introduction: we obtain in this way \eqref{eqBintro}. However, in the proof we shall use greater values of $m$ (see   \S \ref{sec:prop} below). Our main tool will be   a linear recurrence relation satisfied by $F_n^{[s]}(z)$.

\subsection{A linear recurrence relation satisfied by $F_n^{[s]}(z)$} \label{subsec41}

Let $Q_0$, \ldots, $Q_\ell$ and $d_j = \deg(Q_j)$ be as in Lemma \ref{lemBV}. 

\begin{lem} \label{lem:1b}
For any fixed integer $s\ge 1$, the sequence of functions $\big(F_n^{[s]}(z)\big)_{n\ge 1}$ is solution of the inhomogeneous recurrence relation
\begin{equation}\label{eq:3}
\sum_{j=0}^\ell Q_j(-n) F_{n+j}^{[s]}(z)=\sum_{j=0}^\ell\sum_{t=1}^{s-1}\nvP_{j,n,t,s} F_{n+j}^{[t]}(z) 
+ \sum_{j=0}^{\ell} z^{n+j}\nvU_{j,n,s}(\theta)F(z), \quad n\ge 1
\end{equation}
where $\nvP_{j,n,t,s}\in \OK$ and each polynomial $\nvU_{j,n,s}(X)\in \OK[X]$ has degree $\le d_j-s$.

Moreover, letting $ 
\nvU_{j,n,s}(X)=\sum_{q=0}^{d_j-s} \peu_{j,n,s,q} X^q$ the coefficients $\nvP_{j,n,t,s}$ and $\peu_{j,n,s,q}$  are polynomials in $n$, with coefficients in $\OK$ (depending on $j$, $t$, $s$, $q$), such
that
$$\deg(\nvP_{j,n,t,s}) \leq d_j + t-s \mbox{ and } \deg (\peu_{j,n,s,q}) \leq d_j - q - s.$$

 In particular:
 \begin{itemize}
 \item If $s>\mu=\max(d_0,\ldots, d_\ell)$, then $\nvU_{j,n,s}(X)=0$.
 \item If $t < s-d_j$ then $\nvP_{j,n,t,s}=0$. 
 \end{itemize}
\end{lem} 

\begin{proof} We prove \eqref{eq:3} by induction on $s\ge 1$, and this will provide expressions for the various involved quantities. 
In the case $s=1$, we write $Q_j(x)=\sum_{m=0}^{d_j} \rho_{j,m} x^m$ with $\rho_{j,m}\in\OK$ for any $j$, $m$. For any integer $n\ge 1$, we have
\begin{align*}
0=& \int_0^z x^{n-1} LF(x) dx = \sum_{j=0}^\ell\sum_{m=0}^{d_j} \rho_{j,m} \int_0^z x^{n+j-1}(\theta+j)^mF(x) dx
\\
&=\sum_{j=0}^\ell\sum_{m=0}^{d_j} \rho_{j,m} \sum_{p=0}^m \binom{m}{p}j^{m-p} \int_0^z x^{n+j-1}\theta^pF(x) dx
\end{align*}
because $(\theta+j)^m=\sum_{p=0}^m \binom{m}{p}j^{m-p} \theta^p.$ 
After successive integrations by parts (with respect to $\theta$;  all the integrated parts vanish at $x=0$ because $n\ge 1$), we see that 
\begin{multline*}
\int_0^z x^{n+j-1}\theta^pF(x) dx \\
= z^{n+j}\sum_{q=0}^{p-1}(-1)^{p-q-1}(n+j)^{p-q-1} \theta^{q}F(z) +(-1)^{p}(n+j)^p \int_0^z x^{n+j-1}F(x) dx.
\end{multline*} 
Since $\int_0^z x^{n+j-1}F(x) dx = F_{n+j}^{[1]}(z)$ we deduce that 
\begin{multline*}
0=\sum_{j=0}^\ell F_{n+j}^{[1]}(z)\sum_{m=0}^{d_j} \rho_{j,m} \sum_{p=0}^m \binom{m}{p}(-1)^{p}(n+j)^pj^{m-p}
\\
+
\sum_{j=0}^\ell z^{n+j}\sum_{m=0}^{d_j} \rho_{j,m} \sum_{p=0}^m \binom{m}{p}j^{m-p}\sum_{q=0}^{p-1}(-1)^{p-q-1}(n+j)^{p-q-1} \theta^{q}F(z).
\end{multline*}
We now set  for $0\leq q \leq d_j-1$
\begin{equation} \label{eq:15}
\peu_{j,n,1,q}(X)=-\sum_{m=0}^{d_j} \rho_{j,m} \sum_{p=q+1}^m \binom{m}{p}j^{m-p} (-1)^{p-q}(n+j)^{p-q-1} ,
\end{equation}
which is a  polynomial  in $n$ with  coefficients in $\OK$ and degree at most $d_j-q-1$. Therefore $
\nvU_{j,n,1}(X)=\sum_{q=0}^{d_j-1} \peu_{j,n,1,q} X^q 
$
has   degree $\le d_j-1 $ and   coefficients in $\OK$. 
Since   
$$
\sum_{m=0}^{d_j} \rho_{j,m} \sum_{p=0}^m \binom{m}{p}(-1)^{p}(n+j)^pj^{m-p} = \sum_{m=0}^{d_j} \rho_{j,m} (-n)^m= Q_j(-n),
$$
we then deduce that 
$$
\sum_{j=0}^\ell Q_j(-n) F_{n+j}^{[1]}(z)=\sum_{j=0}^{\ell} z^{n+j}\nvU_{j,n,1}(\theta)F(z)
$$
for any integer $n\ge 1$: this proves \eqref{eq:3} for $s=1$.

\medskip

Let us assume that Lemma \ref{lem:1b} holds  for some $s\ge 1$. Then, since  $\int_0^z\frac{1}{x}F_{n+j}^{[s]}(x) dx=F_{n+j}^{[s+1]}(z)$, we have 
\begin{equation*}
\sum_{j=0}^\ell Q_j(-n) F_{n+j}^{[s+1]}(z) 
=\sum_{j=0}^\ell\sum_{t=1}^{s-1}\nvP_{j,n,t,s} F_{n+j}^{[t+1]}(z)
+ \sum_{j=0}^{\ell} \int_0^z x^{n+j-1}\nvU_{j,n,s}(\theta)F(x) dx.
\end{equation*}
Now recall from the case $s=1$ that 
\begin{equation*}
\int_0^z x^{n+j-1}\theta^qF(x) dx 
= z^{n+j}\sum_{h=0}^{q-1}(-1)^{q-h-1}(n+j)^{q-h-1} \theta^{h}F(z) +(-1)^{q}(n+j)^q F_{n+j}^{[1]}(z).
\end{equation*}
Hence, 
\begin{multline*}
\sum_{j=0}^\ell Q_j(-n) F_{n+j}^{[s+1]}(z) 
\\
=\sum_{j=0}^\ell \sum_{t=2}^{s} \nvP_{j,n,t-1,s} F_{n+j}^{[t]}(z)
+ \sum_{j=0}^{\ell} F_{n+j}^{[1]}(z) \Big(\sum_{q=0}^{d_j-s} (-1)^{q}(n+j)^q\peu_{j,n,s,q} \Big)
\\
+\sum_{j=0}^{\ell} z^{n+j}\sum_{q=0}^{d_j-s} \peu_{j,n,s,q}\sum_{h=0}^{q-1}(-1)^{q-h-1}(n+j)^{q-h-1} \theta^{h}F(z).
\end{multline*}
Eq.~\eqref{eq:3} follows for $s+1\ge 2$ with
$$
\nvP_{j,n,t,s+1}=
\begin{cases}
\displaystyle \sum_{q=0}^{d_j-s} (-1)^{q}(n+j)^q \peu_{j,n,s,q} \qquad \textup{for} \quad t=1
\\
\\
\nvP_{j,n,t-1,s} \qquad \textup{for} \quad 2\le t \le s
\end{cases}
$$
and 
$$
\nvU_{j,n,s+1}(X)=\sum_{h=0}^{d_j-s-1}\Big(\sum_{q=h+1}^{d_j-s} (-1)^{q-h-1}(n+j)^{q-h-1}\peu_{j,n,s,q}\Big) X^h.
$$
In particular, 
$$
\peu_{j,n,s+1,h}=\sum_{q=h+1}^{d_j-s}(-1)^{q-h-1}(n+j)^{q-h-1}\peu_{j,n,s,q}, \quad 0\le h\le d_j-s-1.
$$
This completes the proof of Lemma~\ref{lem:1b},  with explicit formulas. 
\end{proof}

\subsection{Proof of Proposition~\ref{prop:1}}\label{sec:prop}

Let $\K$, $F$, $L$  be as in the statement of Proposition~\ref{prop:1}, and $Q_0$, \ldots, $Q_\ell$ be as in Lemma \ref{lemBV}.

To begin with, we claim that if $m$ satisfies \eqref{eqhypm} and Proposition~\ref{prop:1} holds for $m+1$, then  Proposition~\ref{prop:1} holds for $m$. Indeed Lemma \ref{lemBV} asserts that the integer roots of $Q_\ell(-X+\ell)$ are  $\widehat{f}_{1}$, \ldots, $\widehat{f}_{\eta}$
 so that \eqref{eqhypm} yields $Q_\ell(-m)\neq 0$. By induction on  $s\geq 0$, Lemma~\ref{lem:1b} implies that $F_{m+\ell}^{[s]}(z)$ is  a linear combination of $F_{m+j}^{[t]}(z)$ ($0\leq j \leq \ell-1$, $1\leq t \leq s$, with coefficients in $\K$) and $\theta^j F(z)$ ($0\leq j \leq \mu-1$, with coefficients in $\K[z]$ of degree at most $m+\ell$). Then for any $n\geq 1$ we may replace all $F_{m+\ell}^{[s]}(z)$ with this expression in the expansion  \eqref{eq:406} provided by Proposition~\ref{prop:1} with $m+1$. This gives an expansion of the form \eqref{eq:406} with $m$, and the new values of $ \ppP_{j,t,s,n}$ and $ \qqQ_{j,s,n}(z)$ are easily proved to satisfy also $(ii)$ and $(iii)$. This concludes the proof of the claim.

We denote by $\widehat{e}_1, \ldots, \widehat{e}_{\kappa}$ the integer exponents of $L$ at $0$; we have $\kappa\ge 1$  because $LF(z)=0$. Recall that  $\widehat{f}_{1}$, \ldots, $\widehat{f}_{\eta}$ are  the integer exponents of $L$ at  $\infty$, with $\eta=0$ if there isn't any. The claim shows that in proving Proposition~\ref{prop:1} we may assume that $m$ is large; we shall assume from now on that 
\begin{equation} \label{eqhypmgros}
m > - \widehat{e}_{i} \mbox{ and } m > \widehat{f}_{j} - \ell  \; \mbox{ for all } \;1\leq i\le \kappa, \; 1\le j \leq \eta.
\end{equation}
Then we are in the setting of \S \ref{subseclinrec}; in particular, 
  $Q_0(-n)\neq 0$ and $Q_\ell(-n)\neq 0$ for any integer  $n\geq m$.
 As in \S \ref{subseclinrec} we denote by $(u_1(n))_{n\geq m}$, \ldots,  $(u_\ell(n))_{n\geq m}$ a basis of the space of solutions of the homogeneous recurrence relation $\sum_{j=0}^\ell Q_j(-n) U(n+j) = 0$, $n\geq m$, such that $u_j(n)\in\K$ for any $j$ and any $n$. We also define $W(n)$ and $D_j(n)$ as in  \S \ref{subseclinrec}  (see \eqref{eq:8b} and \eqref{eqdjn}). Lemma \ref{lemlinrec} shows that  $\sum_{n = m}^{\infty} \frac{z^n}{W(n)}$,  $\sum_{n=m}^\infty u_j(n)z^n$ and  $\sum_{n=m+1}^\infty \frac{D_j(n)}{Q_\ell(1-n)}z^n$ (with $1\leq j \leq \ell$) are  $G$-functions. Therefore letting $\delta_n >0$ denote a common denominator of the algebraic  numbers  $\frac1{W(k)}$, $\frac{D_j(k)}{ Q_\ell(1-k)}$ ($m+1\leq k \leq n$, $1\leq j \leq \ell$), $u_j(k)$ ($m\leq k \leq n$, $1\leq j \leq \ell$), we have
\begin{equation}\label{eqmajodelta}
\limsup_{n\to\infty} \delta_n^{3/n}\leq C_2(F)
\end{equation}
where $C_2(F)$ is a constant that depends only on $F$. Since $\delta_n\ge 1$, we have $C_2(F)\ge 1$.  For the same reason we have  
\begin{equation}\label{eqmajonorme}
\max_{m+1\leq k \leq n} \max \Big( | u_j(k)|, \, \frac1{| W(k) |} , \, \frac{|  D_j(k)| }{|  Q_\ell(1-k)| } \Big) \leq  C_1(F)^{n(1+o(1))}
\end{equation}
as $n\to\infty$, for any $j\leq \ell$, where $C_1(F)$ is a constant that depends only on $F$.  Increasing $C_1(F)$ if necessary, we may assume that $C_1(F)\geq 1$.

By induction on $s\geq 1$ we shall construct algebraic numbers 
$\ppP_{j,t,s,n}\in \K$ and   polynomials $\qqQ_{j,s,n}(z) \in \K[z]$ of degree at most  $\ell+n-1$ such that for any $n\ge 1$, 
\begin{equation}\label{eq:16}
F_n^{[s]}(z)=\sum_{t=1}^{s}\sum_{j=1}^{\ell+m-1} \ppP_{j,t,s,n}F_j^{[t]}(z) 
+ \sum_{j=0}^{\mu-1} \qqQ_{j,s,n}(z) \theta^j  F(z)
\end{equation}
with the additional properties 
\begin{equation}\label{eqdenoms}
d^s  \delta_{n+s(\ell-1)}^{3s} \ppP_{j,t,s,n'}\in \OK \mbox{ and }d^s  \delta_{n+s(\ell-1)}^{3s}  \qqQ_{j,s,n'}(z) \in \OK[z] \mbox{ for any } n'\leq n
\end{equation}
where $d \geq 1$ depends only on $F$  but neither on $n$ nor on $s$. Together with \eqref{eqmajodelta} this implies assertion $(iii)$ of Proposition~\ref{prop:1}; our construction (in which all formulas are explicit) yields also assertion $(ii)$ using  \eqref{eqmajonorme}.

\medskip

The construction of $\ppP_{j,t,s,n} $ and    $\qqQ_{j,s,n}(z) $ is trivial   if $n \leq \ell + m - 1$: it is enough to choose $\qqQ_{s,j,n}(z)=0$ for any $s, j$, and $\ppP_{j,t,s,n}$ equal to 1 if $j=n$ and $t=s$, equal to 0 otherwise. Therefore we may restrict now to the case $n \geq \ell +m$.

\medskip

We shall prove at the same time the initial step ($s=1$) and the inductive step. With this aim in mind we let $s\geq 0$ and we shall prove the property with $s+1$ (i.e., construct  explicitly $\ppP_{j,t,s+1,n} $ and    $\qqQ_{j,s+1,n}(z)$ such that \eqref{eq:16} and \eqref{eqdenoms} hold); if $s=0$ the proof is unconditional, whereas if $s\geq 1$ the property with $s$ will be used.

\medskip

By Lemma~\ref{lem:1b}, 
the sequence of functions $\big(F_n^{[s+1]}(z)\big)_{n\ge 1}$ is solution of the inhomogeneous recurrence relation
\begin{equation}\label{eq:3b}
\sum_{j=0}^\ell Q_j(-n) F_{n+j}^{[s+1]}(z)=g_{s+1}(n), \quad n \ge 1, 
\end{equation}
with 
\begin{equation}\label{eq55bis}
g_{s+1}(n)=\sum_{j=0}^\ell\sum_{t=1}^{s}\nvP_{j,n,t,s+1} F_{n+j}^{[t]}(z) 
+ \sum_{j=0}^{\ell} z^{n+j}\nvU_{j,n,s+1}(\theta)F(z)
\end{equation}
where $\nvP_{j,n,t,s+1}\in \OK$ and each polynomial $\nvU_{j,n,s+1}(X)\in \OK[X]$ has degree $\le d_j-s-1$.
Lemma~\ref{lemlinrec} shows that   there exist some functions $\ddjspu(z)$ such that for all $n\ge m$, 
\begin{equation}\label{eq:11b}
F_n^{[s+1]}(z)=\sum_{j=1}^\ell \ddjspu(z) u_{j}(n)+\sum_{j=1}^\ell \Big(\sum_{k=m+1}^{n} \frac{\Delta_{s+1,j}(k)}{W(k)}\Big)u_j(n),
\end{equation}
with (using \eqref{eq55bis})
\begin{equation}\label{eq56bis}
\frac{\Delta_{s+1,j}(k)}{W(k)} 
 = \frac{ D_j(k)}{W(k) Q_\ell(1-k)} \Big( \sum_{q=0}^\ell\sum_{t=1}^{s}\nvP_{q,k-1,t,s+1} F_{k-1+q}^{[t]}(z) 
+ \sum_{q=0}^{\ell} z^{k+q-1}\nvU_{q,k-1,s+1}(\theta)F(z)\Big) .
 \end{equation}
The point here is that $F_n^{[s+1]}(z)$,  $g_{s+1}(n)$, $\Delta_{s+1,j}(k)$ depend on $z$, whereas $Q_j(-n)$ does not: the homogeneous recurrence relation \eqref{eq:12b} and  $u_j(n)$, $W(k)$,  $D_j(k)$ do not depend on $z$.
The functions $\ddjspu(z)$ can be determined as follows. We use \eqref{eq:11b} for $n=m, \ldots, m+\ell-1$ so that the linear system of $\ell$ equations  
$$
\sum_{j=1}^\ell \ddjspu(z) u_{j}(n) = F_n^{[s+1]}(z)-\sum_{h=1}^\ell \Big(\sum_{k=m+1}^{n} \frac{\Delta_{s+1,h}(k)}{W(k)}\Big)u_h(n)
$$
is solved by Cram\'er's rule. Indeed, the determinant of the system is $W(m)$   and accordingly a non-zero element of $\K$  by Lemma~\ref{lemlinrec}. Therefore, for any $j$  there   exist some $\alcra\in \K$ (independent of $s$) such that
\begin{equation}\label{eq:17}
\ddjspu(z)=\sum_{p=m}^{\ell+m-1} \alcra \left(F_p^{[s+1]}(z)-\sum_{h=1}^\ell \Big(\sum_{k=m+1}^{p} 
\frac{\Delta_{s+1,h}(k)}{W(k)}\Big)u_h(p)\right).
\end{equation}
Using this equality and  \eqref{eq56bis} into  \eqref{eq:11b} yields, for any $n\ge  m$: 
\begin{equation}\label{eqsomme}
F_n^{[s+1]}(z)= c_1 + c_2 + c_3 + c_4 + c_5, 
\end{equation}
where
\begin{eqnarray*}
c_1 &=& \sum_{j=1}^\ell u_j(n) \sum_{p=m}^{\ell+m-1} \alcra F_p^{[s+1]}(z),
\\
c_2  &=& -\sum_{j=1}^\ell u_j(n) \sum_{p=m}^{\ell+m-1}  \alcra \sum_{h=1}^\ell u_h(p)  \sum_{k=m+1}^{p} \sum_{q=0}^\ell\sum_{t=1}^{s}  \frac{D_h(k)}{W(k)Q_\ell(1-k )} \nvP_{q,k-1,t,s+1} F_{k-1+q}^{[t]}(z) ,
\\
c_3  &=& - \sum_{j=1}^\ell u_j(n)  \sum_{p=m}^{\ell+m-1}  \alcra \sum_{h=1}^\ell u_h(p)  \sum_{k=m+1}^{p}    \frac{D_h(k)}{W(k) Q_\ell(1-k)} \sum_{q=0}^\ell z^{k-1+q} \nvU_{q,k-1,s+1}(\theta)F(z),
\\
c_4  &=&  \sum_{j=1}^\ell u_j(n) \sum_{k=m+1}^{n} \sum_{q=0}^\ell  \sum_{t=1}^{s}   \frac{D_j(k)}{W(k)Q_\ell(1-k)} \nvP_{q,k-1,t,s+1} F_{k-1+q}^{[t]}(z)  ,
\\
c_5  &=&   \sum_{j=1}^\ell u_j(n) \sum_{k=m+1}^{n}   \frac{D_j(k)}{W(k)Q_\ell(1-k)}  \sum_{q=0}^\ell z^{k-1+q} \nvU_{q,k-1,s+1}(\theta)F(z).
\end{eqnarray*}
If $s=0$ then $c_2 $ and $c_4 $ vanish; otherwise we apply \eqref{eq:16} with each $t\in \{1,\ldots,s\}$ and get
\begin{eqnarray*}
c_2 &=&  -\sum_{j=1}^\ell u_j(n) \sum_{p=m}^{\ell+m-1}  \alcra \sum_{h=1}^\ell u_h(p)  \sum_{k=m+1}^{p} \sum_{q=0}^\ell\sum_{t=1}^{s}  \frac{D_h(k)}{W(k)Q_\ell(1-k )} \nvP_{q,k-1,t,s+1}
\\
&&\Big(  \sum_{t'=1}^{t}\sum_{j'=1}^{\ell+m-1} \ppP_{j',t',t,k-1+q}F_{j'}^{[t']}(z) 
+ \sum_{j'=0}^{\mu-1} \qqQ_{j',t,k-1+q}(z) \theta^{j'}  F(z)
\Big) , 
\\
c_4 &=&   \sum_{j=1}^\ell u_j(n) \sum_{k=m+1}^{n} \sum_{q=0}^\ell  \sum_{t=1}^{s}   \frac{D_j(k)}{W(k)Q_\ell(1-k)} \nvP_{q,k-1,t,s+1}  
\\
&&\Big(  \sum_{t'=1}^{t}\sum_{j'=1}^{\ell+m-1} \ppP_{j',t',t,k-1+q}F_{j'}^{[t']}(z) 
+ \sum_{j'=0}^{\mu-1} \qqQ_{j',t,k-1+q}(z) \theta^{j'}  F(z) 
\Big).
\end{eqnarray*}
We shall now define the coefficients $ \ppP_{p,t',s+1,n} $ and $ \qqQ_{j',s+1,n}(z)$ in such a way that \eqref{eqsomme} reads
$$
F_n^{[s+1]}(z)=\sum_{t'=1}^{s+1}\sum_{j'=1}^{\ell+m-1} \ppP_{j',t',s+1,n}F_{j'}^{[t']}(z) 
+ \sum_{j'=0}^{\mu-1} \qqQ_{j',s+1,n}(z)\theta^{j'} F(z).
$$
Taking $c_1$ into account we let 
$$\ppP_{p, s+1, s+1, n} = \left\{ \begin{array}{l} 0 \mbox{ if } 1 \leq p \leq m-1, \\ 
\sum_{j=1}^\ell \alcra u_j(n)  \mbox{ if } m \leq p \leq \ell+m-1.
\end{array}\right.$$
If $s\geq 1$ then considering $c_2 $ and $c_4$ we let, for any $t'$, $j'$ such that $1\leq t' \leq s$ and $1\leq j'\leq \ell+m-1$:
\begin{multline*}
\ppP_{j', t',   s+1, n} 
=\\
  -\sum_{j=1}^\ell u_j(n) \sum_{p=m}^{\ell+m-1}  \alcra \sum_{h=1}^\ell u_h(p)  \sum_{k=m+1}^{p} \sum_{q=0}^\ell\sum_{t=t'}^{s}  \frac{D_h(k)}{W(k)Q_\ell(1-k )} \nvP_{q,k-1,t,s+1}  \ppP_{j',t',t,k-1+q}
\\
+   \sum_{j=1}^\ell u_j(n) \sum_{k=m+1}^{n} \sum_{q=0}^\ell  \sum_{t=t'}^{s}   \frac{D_j(k)}{W(k)Q_\ell(1-k)} \nvP_{q,k-1,t,s+1}  \ppP_{j',t',t,k-1+q}.
\end{multline*}
Now recall that we assume $n\geq \ell+m$. Then in each term of the sum we have $k-1+q\leq n+\ell-1$ so that \eqref{eqdenoms} yields $d^s \delta_{n+(s+1)(\ell-1)}^{3s}  \ppP_{j',t',t,k-1+q} \in\OK$. By definition of $\delta_{n+(s+1)(\ell-1)}$ we obtain
 in both cases ($s=0$ and $s\geq 1$) that
$$d^{s+1} \delta_{n+(s+1)(\ell-1)}^{3(s+1)} \ppP_{j', t',   s+1, n'} \in \OK \mbox{ for any } n'\leq n, \mbox{  any }1 \leq j' \leq \ell+m-1 \mbox{ and any } 1\leq t'\leq s+1,$$
where $d \geq 1$ is chosen (in terms of $F$  only, independently from $n$ and $s$) such that
$$d  \alcra   u_h(p) \in \OK \mbox{ for any } m \leq p \leq  \ell+m-1 \mbox{ and any }  1\leq j,h \leq \ell.$$

On the other hand, writing  $\nvU_{j,k,s+1}(X) = \sum_{q=0}^{\mu-1} \peu_{j,k,s+1,q} X^q$ (so that $\peu_{j,k,s+1,q} = 0$ if $\deg \nvU_{j,k,s+1} < q \leq \mu-1$) and considering the coefficients of $\theta^{j'}$ in $c_3$, $c_5$, $c_2$ and $c_4$ we let for any $j'$ with $0\leq j' \leq \mu-1$:
\begin{align*}
 &\qqQ_{j',s+1,n}(z)
 =   - \sum_{j=1}^\ell u_j(n)  \sum_{p=m}^{\ell+m-1}  \alcra \sum_{h=1}^\ell u_h(p)  \sum_{k=m+1}^{p}    \frac{D_h(k)}{W(k) Q_\ell(1-k)} \sum_{q=0}^\ell z^{k-1+q} \peu_{q,k-1,s+1, j'}
  \\
&\quad +   \sum_{j=1}^\ell u_j(n) \sum_{k=m+1}^{n}   \frac{D_j(k)}{W(k)Q_\ell(1-k)}  \sum_{q=0}^\ell z^{k-1+q} \peu_{q,k-1,s+1, j'}
\\
&\quad -    \sum_{j=1}^\ell u_j(n) \sum_{p=m}^{\ell+m-1}  \alcra \sum_{h=1}^\ell u_h(p)  \sum_{k=m+1}^{p} \sum_{q=0}^\ell\sum_{t=1}^{s}  \frac{D_h(k)}{W(k)Q_\ell(1-k )} \nvP_{q,k-1,t,s+1}
 \qqQ_{j',t,k-1+q}(z)
\\
&\quad +     \sum_{j=1}^\ell u_j(n) \sum_{k=m+1}^{n} \sum_{q=0}^\ell  \sum_{t=1}^{s}   \frac{D_j(k)}{W(k)Q_\ell(1-k)} \nvP_{q,k-1,t,s+1}   \qqQ_{j',t,k-1+q}(z).
\end{align*}
Then we have
$$d^{s+1} \delta_{n+(s+1)(\ell-1)}^{3(s+1)}  \qqQ_{j',s+1,n'}(z) \in \OK[X] \mbox{ for any } n'\leq n  \mbox{ and  } \deg (\qqQ_{j',s+1,n})  \leq n + (s+1)(\ell-1)$$
 for any $ j'$.

At last assertion $(ii)$ of Proposition \ref{prop:1} follows also from these formulas, using Lemma \ref{lem:1b} and \eqref{eqmajonorme}.

\section{Proof of Theorem~\ref{theo:1}} \label{secpreuvethun}

In this section, we introduce a power series that will play the usual role of an auxiliary function in transcendance theory.  We denote by $R>0$ the  radius of convergence of $F(z)=\sum_{k=0}^\infty A_k z^k$.

\medskip

Let $r, S\ge 0 $ be integers such that $r\le S$. Let us define the following auxiliary series, for $n\geq 0$: 
\begin{align*}
T_{S,r,n}(z)& = n!^{S-r}\sum_{k=0}^\infty \frac{k(k-1)\cdots (k-rn+1)}{(k+1)^S(k+2)^S\cdots (k+n+1)^S} \,A_k \,z^{-k}
\\ &= n!^{S-r}\sum_{k=0}^\infty \frac{(k-rn+1)_{rn}}{(k+1)_{n+1}^S}\, A_k \,z^{-k}.
\end{align*}
It converges for any $z$ such that $\vert z\vert >1/R.$ If $A_k=1$ for all $k\ge0$, we recover  the series $N_n(z)$ in \cite{ribordeaux}, up to a factor of $z.$

As in \S \ref{secpreuveprinc} we  let $\theta =  z \frac{d}{dz}$, and as in the introduction we let $\lz = \max(\ell, \widehat f_1,\ldots,\widehat f_\eta)$ where 
$\widehat{f}_{1}$, \ldots, $\widehat{f}_{\eta}$ are the integer exponents of $L$ at  $\infty$  and $\ell$ is defined as in Lemma \ref{lemBV}.

\subsection{A linear form}
We now make the connection between $T_{S,r,n}(z)$ and the functions $F_{n}^{[s]}(z)$.

\begin{lem} \label{lem600} Let us assume that $n\ge \lz$. There exist some polynomials $C_{u,s,n}(X)\in \K[X]$ and $\widetilde{C}_{u,n}(X)\in \K[X]$ of respective 
degrees $\le n+1$ and $\le  n+1 +S(\ell-1)$ such that, 
for any $z$ such that $\vert z\vert >1/R$, we have
$$
T_{S,r,n}(z)=\sum_{u=1}^{\lz}\sum_{s=1}^S C_{u,s,n}(z) F_{u}^{[s]}(1/z)+  \sum_{u=0}^{\mu-1} \widetilde{C}_{u,n}(z) z^{-S(\ell-1)} (\theta^uF)(1/z).
$$
\end{lem}
\begin{Remark} Since the Taylor expansion of $T_{S,r,n}(z)$ at $z=\infty$ has order $\ge rn+1$, this lemma shows that $T_{S,r,n}(z)$ can  be interpreted has an explicit Pad\'e-type approximant at $z=\infty$ for the functions $F_{u}^{[s]}(1/z)$ and $(\theta^uF)(1/z)$.
\end{Remark}
\begin{proof}
We have the partial fractions expansion in $k$: 
\begin{equation}\label{eq:100}
n!^{S-r}\frac{k(k-1)\cdots (k-rn+1)}{(k+1)^S(k+2)^S\cdots (k+n+1)^S} = \sum_{j=1}^{n+1} \sum_{s=1}^S \frac{c_{j,s,n}}{(k+j)^s}
\end{equation}
for some $c_{j,s,n}\in \mathbb Q$, which also depend on $r$ and $S$. It follows that 
$$
T_{S,r,n}(z)=\sum_{j=1}^{n+1} \sum_{s=1}^S c_{j,s,n} z^{j}F_{j}^{[s]}(1/z).
$$
Since  $n\ge \lz$, by Proposition~\ref{prop:1} (with $m = \lz - \ell+1$) we have
\begin{align*}
T_{S,r,n}(z)&=\sum_{j=1}^{\lz}\sum_{s=1}^S c_{j,s,n}  z^{j}F_{j}^{[s]}(1/z) + \sum_{j=\ell+m}^{n+1} \sum_{s=1}^S c_{j,s,n}  z^{j}F_{j}^{[s]}(1/z)
\\
&=\sum_{j=1}^{\lz}\sum_{s=1}^S c_{j,s,n}  z^{j}F_{j}^{[s]}(1/z) \\
&\qquad + \sum_{j=\ell+m}^{n+1} \sum_{s=1}^S c_{j,s,n}z^j  \left(\sum_{t=1}^{s}\sum_{u=1}^{\lz} \ppP_{u,t,s,j}F_u^{[t]}(1/z) 
+ \sum_{u=0}^{\mu-1} \qqQ_{u,s,j}(1/z)(\theta^u F)(1/z)\right)
\\
&=\sum_{u=1}^{\lz}\sum_{s=1}^S C_{u,s,n}(z) F_{u}^{[s]}(1/z)+  \sum_{u=0}^{\mu-1} \widetilde{C}_{u,n}(z)  z^{-S(\ell-1)}  (\theta^uF)(1/z) 
\end{align*}
where  
\begin{equation}\label{eq:101}
C_{u,s,n}(z) = c_{u,s,n}z^{u}+\sum_{j=\ell+m}^{n+1}\sum_{\sigma=s}^S z^jc_{j,\sigma,n}\ppP_{u,s,\sigma,j}
\end{equation}
and 
\begin{equation}\label{eq:102}
\widetilde{C}_{u,n}(z) = \sum_{j=\lz+1}^{n+1} \sum_{s=1}^S c_{j,s,n} z^{j+S(\ell-1)}\qqQ_{u,s,j}(1/z).
\end{equation}
The assertion on the degree of these polynomials is clear from their expressions since $\qqQ_{u,s,j}(1/z)$ is a polynomial in $1/z$ of degree at most $j+s(\ell-1)$.
\end{proof}

\subsection{Analytic and arithmetic bounds for $C_{u,s,n}(z)$ and $\widetilde{C}_{u,n}(z)$}

In this section, we prove two lemmas concerning the coefficients of the polynomials $C_{u,s,n}(z)$ and $\widetilde{C}_{u,n}(z)$. Given $\xi\in\Qbar$ we denote by $\house{\xi}$ the house of $\xi$, i.e. the maximum modulus of the Galois conjugates of $\xi$.

\begin{lem}\label{lem:6} For any $z\in \Qbar$, we have
$$
\limsup_{n\to +\infty} \big(\max_{u,s}\house{ C_{u,s,n}(z)}\big)^{1/n}\le C_1(F)^Sr^r2^{S+r+1}\max(1, \house{z})
$$
and
$$
\limsup_{n\to +\infty} \big(\max_{u} \house{  \widetilde C_{u,n}(z) } \big)^{1/n}\le C_1(F)^Sr^r2^{S+r+1}\max(1, \house{z}).
$$
\end{lem}
\begin{proof} In~\cite[Lemma 4]{ribordeaux}, it is proved that the coefficients $c_{j,s,n}$ in \eqref{eq:100} satisfy
$$
\vert c_{j,s,n}\vert \le (rn+1)2^S (r^r 2^{S+r+1})^n
$$
for all $j,s,n$. 
(Our $c_{j,s,n}$ are noted $c_{s,j-1,n}$  in \cite{ribordeaux}). 
To conclude the proof, we simply use this bound  in \eqref{eq:101} and 
\eqref{eq:102} together with Proposition~\ref{prop:1}$(ii)$.
\end{proof}

\begin{lem} \label{lem:6bis}  Let $z\in\K$ and $q\in\N\etoile$ be such that $qz\in\OK$. Then there exists a sequence $(\Delta_n)_{n\geq 1}$ of positive rational integers 
  such that  for any $ u$, $s$: 
$$
\Delta_nC_{u,s,n}(z) \in \OK,  \quad  \Delta_n\widetilde{C}_{u,n}(z) \in \OK  ,   \quad \textup{and}  \quad 
\lim_{n\to +\infty} \Delta_n^{1/n} =  q  C_2(F)^S e^{S}.
$$
\end{lem}
\begin{proof} Let $d_n=\textup{lcm}\{1,2,\ldots,n\}$. The proof of~\cite[Lemme 5]{ribordeaux} shows  that $d_n^S c_{j,s,n}\in \mathbb Z$ for all $j,s,n$; we recall that $\lim_n d_n^{1/n}=e$. 
On the other hand, in  Proposition~\ref{prop:1}$(iii)$  we may assume that $D(F,S,n)\geq C_2(F)^{Sn}/2$, upon multiplying $D(F,S,n)$ with a suitable positive integer if necessary, so that $\lim D(F,S,n)^{1/n} = C_2(F)^{S}$. Then
 the result follows again from  \eqref{eq:101} and \eqref{eq:102}.
\end{proof}

\subsection{Asymptotic estimate of the linear form} \label{subsec53}

The following lemma will be proved in \S \ref{sec:asympT} using singularity analysis and  the saddle point method.

\begin{lem} \label{lem:7nv} Let $\alpha\in\C$ be such that $0< | \alpha| < R$. Assume that $S$ is sufficiently large (with respect to $F$ and $\alpha$), and that $r$ is the integer part of $\frac{S}{(\log S)^2}$. Then there exist some 
 integers $Q\geq 1$ and $\lambda\geq 0$, real numbers $a  $ and $\kappa$, non-zero complex numbers $c_1$,\ldots, $c_Q$, and pairwise distinct complex numbers $\zeta_1$, \ldots, $\zeta_Q$, such that $|\zeta_q|=1$ for any $q$, 
$$
T_{S,r,n}(1/\alpha ) = a^n n^\kappa \log (n)^\lambda \Big( \sum_{q=1}^Q c_q \zeta_q^n +o(1)\Big) \mbox{ as } n\to\infty,
$$
and
$$0 < a  \le \frac1{r^{S-r}}.$$
\end{lem}

\subsection{Completion of the proof of Theorem~\ref{theo:1}}\label{sec:completion}

Let $\alpha$ be a non-zero element of $\K$ such   that $| \alpha| < R$; choose $q\in\N\etoile$ such that $q/\alpha \in\OK$.  By Lemmas~\ref{lem:6} and~\ref{lem:6bis}, $p_{u,s,n}:=\Delta_nC_{u,s,n}(1/\alpha)$ and 
$\tilde p_{u,n}:=\Delta_n\widetilde{C}_{u,n}(1/\alpha)$  belong to $\OK$ and for any $u$, $s$ we have 
$$
\limsup_{n\to +\infty} \max(\vert p_{u,s,n} \vert^{1/n}, \vert \tilde p_{u,n} \vert^{1/n}) \le b:=q  C_1(F)^SC_2(F)^Se^S r^r2^{S+r+1}\max(1, \house{1/\alpha}).
$$
Using Lemma \ref{lem600} we consider 
$$
\tau_n:= \Delta_n T_{S,r,n}( 1/\alpha )= \sum_{u=1}^{\lz}\sum_{s=1}^S p_{u,s,n} F_u^{[s]}(\alpha) + \sum_{u=0}^{\mu-1} \tilde{p}_{u,n} \alpha^{S(\ell-1)} (\theta^u F)(\alpha).
$$
Choosing $r = [ S /(\log S)^{2}]$,  Lemmas \ref{lem:6bis} and~\ref{lem:7nv} yield as $ n\to\infty$:
  $$ \tau_n = a_0^{n(1+o(1))}  \Big( \sum_{q=1}^Q c_q \zeta_q^n +o(1)\Big) \mbox{ with } 0 < a_0  < \frac{q C_2(F)^S e^S}{r^{S-r}}.$$

Let $\Psi_{ \alpha,S}$ denote the $\K$-vector space spanned by the numbers $F_u^{[s]}( \alpha)$ and $(\theta^v F)( \alpha)$, $1\le u\le \lz$, $1\le s\le S$, $0\le v\le \mu-1$.  
It follows from Theorem \ref{th:nest}   that 
$$
\dim_{\K} (\Psi_{ \alpha,S})\ge \frac{1}{[\K:\Q]} \Big(1 - \frac{\log (a_0) }{\log (b)}\Big).$$
Now, as $S\to +\infty$, 
$$\log(b)  = \log(2eC_1(F)C_2(F))S+o(S) \mbox{ and } 
\log(a_0) \leq  -S\log(S)+ o(S\log S)$$
so that 
\begin{equation}\label{eq:300}
\dim_{\K} (\Psi_{ \alpha,S})\ge  \frac{1+o(1)}{  [\K:\Q] \log(2e C_1(F)C_2(F))}\log(S)
\end{equation}
as $S\to +\infty$. 

We recall that $\Phi_{ \alpha,S}$ is the $\K$-vector space spanned by the numbers $F_u^{[s]}( \alpha)$ for  $u\ge 1$ and $0\le s\le S$. Now, \eqref{eq:406} in Proposition~\ref{prop:1} with $m =\lz-\ell+1$ (i.e. \eqref{eqBintro}) and $z= \alpha$ shows that in fact $\Phi_{ \alpha,S}$ is a  $\K$-subspace of $\Psi_{ \alpha,S}$. In particular, for any $S\ge 0$, 
$$
\dim_{\K} (\Phi_{ \alpha,S})\le \dim_{\K} (\Psi_{ \alpha,S}) \le \lz S+\mu,  
$$
which proves the right-hand side of~\eqref{eq:200} in Theorem~\ref{theo:1}. On the other hand, we also have  
$$
\dim_{\K} (\Psi_{ \alpha,S})\le \dim_{\K} (\Phi_{ \alpha,S}) +\mu
$$ 
so that the lower bound \eqref{eq:300} holds as well with  $\Phi_{ \alpha,S}$ instead of  $\Psi_{ \alpha,S}$ because $\mu$ is independent from $S$. This proves the left hand side of~\eqref{eq:200} in Theorem~\ref{theo:1} with $C(F)=\log(2e C_1(F)C_2(F))$.

\section{Asymptotic behavior of  $T_{S,r,n}(1/\alpha)$}\label{sec:asympT}

In this long section, we determine the precise asymptotic behavior of $T_{S,r,n}(1/\alpha)$ as $n\to +\infty$, under certain conditions on $r$ and $S$. The result is presented as 
Proposition~\ref{prop:asympT} at the very end of the section, and then we deduce from it Lemma \ref{lem:7nv} stated in \S \ref{subsec53}. 
Before that, we state and prove many preliminary results.

\subsection{Analytic representation of $T_{S,r,n}(1/\alpha)$}

Let $\alpha\in\C $ be such that $0< \vert \alpha\vert<R$. We start with 
\begin{align*}
T_{S,r,n}(1/\alpha)
&= n!^{S-r}\sum_{k=0}^\infty \frac{(k-rn+1)_{rn}}{(k+1)_{n+1}^S}\, A_k \,\alpha^{k}
\\
&=n!^{S-r}\sum_{k=0}^\infty \frac{(k-rn+1)_{rn}}{(k+1)_{n+1}^S}\left(\frac{1}{2i\pi}\int_{\mathcal C}\frac{F(z)}{z^{k+1}}dz\right) \alpha^k
\end{align*}
where $\mathcal{C}$ is any direct closed path surrounding $0$ and enclosing none of the singularities of $F(z)$. We want to define a suitable  analytic  function $A(z)$ such that $A(z)=A_k$ for any large enough integer $k$.

Let $\xi_1, \ldots, \xi_p$ denote the finite singularities of $F(z)$. We exclude from this list 
possible removable singularities which contribute $0$ to~\eqref{eq:600} below; then $\xi_j\neq 0$ for any $j$. We have $p\ge 1$ because $F(z)$ is not a polynomial. Amongst these singularities, we will not distinguish   poles from  branch points. 

Let $\vartheta\in(\frac{3\pi}4, \pi)$ be such that $\arg(\xi_j)\not\equiv \vartheta\bmod 2\pi$ for any $j$. We choose also $\eps_j\in (\frac{-\pi}8, \frac{\pi}8)$ such that the half-lines $L_j = \xi_j + \xi_j e^{i\eps_j} \R_+$ are pairwise disjoint, and disjoint from $L_0 = e^{i\vartheta}\R_+$; note that $\pi/8$ (and $3\pi/4$ above) do not play a special role here. Then  
 $\mathcal D=\mathbb{C}\setminus (L_0 \cup L_1\cup \cdots \cup L_p)$ is simply connected; using analytic continuation $F$ is well-defined on  $\mathcal D$. Moreover for $z\in\C\setminus L_0$ we choose the value of $\arg (z)$ between $\vartheta-2\pi$ and $\vartheta$ so that $\log(z) = \ln|z| + i \arg(z)$ is also well-defined on  $\mathcal D$. Unless otherwise stated, we shall use this choice everywhere until the end of the proof of Lemma \ref{lem:Tserint}.

Since $F(z)$ is fuchsian, it has moderate growth at $\infty$, i.e there exists $u>0$ such that $\vert F(z)\vert \ll \vert z\vert^u$ as $z\to \infty, z\in \mathcal{D}$. Hence if $k>u$, we can ``send'' $\mathcal{C}$ to $\infty$: we then have
$$
\frac{1}{2i\pi}\int_{\mathcal C}\frac{F(x)}{x^{k+1}}dx = \sum_{j=1}^p \frac{1}{2i\pi}\int_{\widehat{L}_j}\frac{F(x)}{x^{k+1}}dx
$$
where for each $j$, $\widehat{L}_j$ is a Hankel contour: from $\infty$ to $\xi_j$ on one bank of the cut $L_j$ (namely with $\arg(z-\xi_j)$ slightly less than $\arg(\xi_j)+\eps_j$) and back to $\infty$ on the other bank, and always at a (constant) positive distance of $L_j$. 
We thus have the representation
\begin{equation}\label{eq:600}
A_k=\sum_{j=1}^p \frac{1}{2i\pi}\int_{\widehat{L}_j} \frac{F(x)}{x^{k+1}} dx.
\end{equation}
Note that if $\xi_j$ is a pole of $F(z)$, then 
$$
\frac{1}{2i \pi}\int_{\widehat{L}_j} \frac{F(x)}{x^{k+1}} dx = \text{Res}\Big(\frac{F(x)}{x^{k+1}}, x=\xi_j\Big).
$$
We define 
$$
B_j(z)=\frac{1}{2i \pi} \int_{\widehat{L}_j} \frac{F(x)}{x^{z+1}} dx = \frac{\xi_j^{-z}}{2i\pi}\int_{\widetilde{L}_j} \frac{F(\xi_j x)}{x^{z+1}} dx
$$
where $\widetilde{L}_j=\xi_j^{-1}\widehat{L}_j$; recall that $-\frac{\pi}8 < \eps_j < \frac{\pi}8$ so that $\arg(\xi_j x)  =\arg(\xi_j) + \arg(x)$ when $x$ lies on $\widetilde{L}_j$.
 Each function $B_j(z)$ is analytic in $\textup{Re}(z)>u$ (at least).
Note that $\widetilde{L}_j$ is again a Hankel contour: from $\infty$ to $1$ on the bank of the cut $ 1 + e^{i\eps_j}\R_+ $ where $\arg(x-1)$ is slightly less than $\eps_j$, and back to $\infty$ on the other bank,  always at a (constant) positive distance of the cut.

\begin{lem}\label{lem:new1}  
$(i)$ The function 
$
A(z):=\sum_{j=1}^p B_j(z)
$
is analytic in $\textup{Re}(z)>u$ and $A(k)=A_k$ for any integer $k>u$.

$(ii)$ For each $j$, there exist $s_j\in \mathbb N$, $\beta_j\in \mathbb Q$ and $\kappa_j\in \mathbb C\setminus \{0\}$ such that for any $t$ such that $\textup{Re}(t)>0$,  
\begin{equation}\label{eq:602}
B_j(tn) = \kappa_j \frac{\log(n)^{s_j}}{(tn)^{\beta_j }\xi_j^{tn}}\Big(1+\mathcal{O}\big(\frac{1}{\log(n)}\big)\Big)
\end{equation}
as $n\to +\infty$. The implicit constant is uniform in any half-plane $\textup{Re}(t)\ge d$ where $d$ is a  fixed positive constant. 
\end{lem}

\begin{proof} Item $(i)$ is clear. Item $(ii)$ is standard but we sketch the argument for the reader's convenience; it is essentially the same one as 
in the proof of~\cite[Theorem 3]{rrjnt}. We fix $j\in\{1,\ldots,p\}$.
Given $x\in\C\setminus  \widetilde{L}_j $ we choose the value of $\arg (1-x)$ between $\eps_j -\pi$ and $\eps_j +\pi$; then $ \log(1-x)$ and $(1-x)^{t} $ are well-defined (for any $t\in\C$). To make things more precise we shall write $\log_j$ when we refer to this choice, and $\log$ when the previous one is used.
By the Andr\'e-Chudnovsky-Katz Theorem,  in a neighborhood of $x=1$, $x\not\in \widetilde{L}_j$, we have
\begin{equation}\label{eq:601}
F(\xi_j x)=\sum_{s\in S_{j}}\sum_{t\in T_{j}} \kappa_{j,s,t} \log_j(1-x)^{s} (1-x)^{t} F_{j,s,t}(1-x)
\end{equation}
where $\kappa_{j,s,t}\in \mathbb C$, $S_{j}\subset \mathbb N$, $T_j\subset \mathbb Q$ and $F_{s,t,j}(x)$ are $G$-functions. 
(In fact, the full force of the Andr\'e-Chudnovsky-Katz  Theorem is not needed here: the theory of fuchsian equations ensures that 
\eqref{eq:601} holds {\em a priori} with $T_j\subset \Qbar$ and $F_{j,s,t}(x)$ holomorphic at $x=0$, which is enough.) Each function 
$F_{s,t,j}(x)$ can be analytically continued but we would like to use only its Taylor series around $x=0$. To do that, 
we now use a classical trick that goes back to N\"orlund~\cite{norlund} at least. We set $x=1/(1-y)^{\omega}$, where $\omega>0$ is a parameter to be specified below, so that
\begin{align}
B_j(z)&= \frac{\xi_j^{-z}}{2i\pi}\int_{\widetilde{L}_j} \frac{F(\xi_j x)}{x^{z+1}} dx\nonumber
\\
&=\frac{\omega\xi_j^{-z}}{2i\pi}\int_{M_j} F\left(\frac{\xi_j}{(1-y)^{\omega}}\right)(1-y)^{z\omega-1} dy \label{eqbjarg}
\end{align}
where $M_j$ is a  closed loop around $N_j$,  with negative orientation,  passing through $1$; here $N_j$ is the set of all $y = 1 - (1+e^{i\eps_j}R)^{-1/\omega}$ with $R\in \R_+$. It is a cut going from 1 to 0, and if $\eps_j = 0$ (which is a suitable choice if $\arg (\xi_i) \not\equiv \arg (\xi_j) \bmod 2\pi$ whenever $i,j\in\{1,\ldots,p\}$ are distinct) then $N_j$ is the real interval $ [0,1]$. We may assume that $\textup{Re}(y )\leq 1$ for any $y\in M_j$ so that $\log(1-y)$ is well-defined for any $y\in M_j \setminus \{1\}$, and also $(1-y)^{z\omega-1}$. On the other hand, in the integral  \eqref{eqbjarg} we have $y\not\in N_j$ so that $(1-y)^{-\omega}\not\in   \widetilde{L}_j$: we have defined $\log_j(1-(1-y)^{-\omega})$ and we use it   in what follows. 

 We have
$$
F\left(\frac{\xi_j}{(1-y)^{\omega}}\right)=\sum_{s\in S_{j}}\sum_{t\in T_{j}} \kappa_{j,s,t} \frac{\partial^{s}}{\partial \varepsilon^{s}}
\left(
\left(1-\frac{1}{(1-y)^{\omega}}\right)^{t+\varepsilon} F_{j,s,t}\left(1-\frac{1}{(1-y)^{\omega}}\right)\right)_{\varepsilon=0}.
$$
We now choose $\omega$ small enough such that     $N_j$ is strictly inside the disk of convergence of each of the series 
$$
\left(1-\frac{1}{(1-y)^{\omega}}\right)^{t+\varepsilon} F_{j,s,t}\left(1-\frac{1}{(1-y)^{\omega}}\right)=y^{t+\varepsilon}\sum_{m=0}^\infty 
\phi_{j,s,t,m}(\varepsilon, \omega) y^m
$$
for any $\varepsilon>0$, where the coefficients $\phi_{j,s,t,m}(\varepsilon, \omega)$ are infinitely differentiable at $\varepsilon=0$. Here $\log (y)$ is defined with a cut along $N_j \cup (1+\R_+)$; if $y$ does not lie on this cut then $\eps_j < \arg (y) < \eps_j + 2\pi$.  Since  we may also ensure that $M_j$ is strictly inside these disks, 
  we can exchange summation and integral and we obtain 
$$
B_j(z) =\omega \xi_j^{-z} \sum_{s\in S_{j}}\sum_{t\in T_{j}} \kappa_{j,s,t}\sum_{m=0}^\infty \frac{\partial^{s}}{\partial \varepsilon^{s}}
\left(\phi_{j,s,t,m}(\varepsilon,\omega) 
\frac{1}{2i\pi}\int_{M_j} y^{m+t+\varepsilon}(1-y)^{z\omega-1}dy\right)_{\varepsilon=0}.
$$
Now this integral can be computed in terms of Euler's Beta function $B(z_1,z_2) = \frac{\Gamma(z_1)\Gamma(z_2)}{\Gamma(z_1+z_2)}$ as follows. Using the residue theorem we may assume that $\eps_j=0$, i.e. $N_j = [0,1]$. If $t > -1$ then $M_j$ can be taken as the succession of a path from 1 to 0 along this segment (in which $\arg (y) =2\pi$) and a path from 0 to 1 along the same segment (but in which $\arg(y)= 0$); in both paths we have $\arg(1-y)=0$. Therefore we obtain:
$$
  \int_{M_j} y^{m+t+\varepsilon}(1-y)^{z\omega-1}dy = \Big( 1 - e^{2i\pi(m+t+\varepsilon)}\Big) \frac{\Gamma(m+t+\varepsilon+1)\Gamma(\omega z)}{\Gamma(\omega z+m+t+\varepsilon+1)}.$$
Using analytic continuation with respect to $t$, we see that this equality holds for any $t\in\C$.
Hence,  using the reflection formula we obtain 
\begin{equation}\label{eq:605}
B_j(z) =\omega \xi_j^{-z} \sum_{s\in S_{j}}\sum_{t\in T_{j}} \kappa_{j,s,t}\sum_{m=0}^\infty \frac{\partial^{s}}{\partial \varepsilon^{s}}\left(
\frac{\phi_{j,s,t,m}(\varepsilon,\omega)e^{i\pi(m+t+\varepsilon)}\Gamma(\omega z)}{\Gamma(\omega z+m+t+\varepsilon+1)\Gamma(-m-t-\varepsilon)} \right)_{\varepsilon=0},
\end{equation}
where all the involved series are absolutely convergent; they are called ``s\'eries de facult\'es'' in \cite{norlund}. 
Note that of course the result does not depend on the chosen value of $\omega$ (but convergence holds only if $\omega$ is small enough). 

Convergent ``S\'eries de facult\'es'' play a role similar to asymptotic expansions (except that usually the latter are divergent): roughly speaking, instead of asymptotic expansions with terms of the form $1/z^m$, we obtain convergent expansions with terms of the form $1/(z)_m$. The asymptotic expansion~\eqref{eq:602} follows by classical arguments because we can easily get the asymptotic expansion of a ``s\'erie  de facult\'es'' as $z\to \infty$: if we differentiate $s$ times $1/(z)_m=\Gamma(z)/\Gamma(z+m)$ with respect to $m$, we obtain a finite sum of terms involving (derivatives of) the  Digamma function $\Psi(z)=\Gamma'(z)/\Gamma(z)$, which are asymptotically of the form $\log(z)^t/z^m$ with $0\le t \le s$. See~\cite{rrjnt} for details when $s=0$ and \cite[pp. 42--45]{norlund} for the general case, especially Th\'eor\`eme~1 there. 

Moreover,
the constant $\kappa_j$ in~\eqref{eq:602}  is non-zero. Indeed since $\xi_j$ is a non-removable singularity of $F(z)$,  the overall asymptotic expansion of $B_j(tn)$ obtained from~\eqref{eq:605} cannot be identically~$0$ as $n\to +\infty$.
\end{proof}

\medskip

In what follows we let 
$$ \grosb :=  \int_{c-i \infty }^{c+i \infty}  B_j(tn)
\frac{n!^{S-r}\Gamma((r-t)n)\Gamma(tn+1)^{S+1}}{\Gamma((t+1)n+2)^S}  \left(-\alpha\right)^{tn} dt$$
for $1\leq j \leq p$, where $c$ is such that $0 < c < r$; the residue theorem shows that $\grosb$ is independent from the choice of $c$.

\begin{lem} \label{lem:Tserint} If  $0<\vert \alpha\vert <R$ and  $r>u$
then for $n$ large enough, we have
\begin{equation}\label{eq:607}
T_{S,r,n}(1/\alpha)
=\sum_{j=1}^p \frac{(-1)^{rn}n}{2i\pi} \grosb.
\end{equation}
\end{lem}
\begin{proof}
Let  $\mathcal{R}_{N,c}$  denote the positively oriented  rectangular contour  with vertices $c n  \pm iN$ and $N+\frac{1}{2}\pm iN$, where $u<c<r$ and the integer $N$ is such that $N \ge rn$. Then by the residue theorem
\begin{align*}
n!^{S-r}\sum_{k=rn}^N \frac{(k-rn+1)_{rn}}{(k+1)_{n+1}^S} A_k \alpha^k 
&= \frac{n!^{S-r}}{2i\pi}\int_{\mathcal{R}_{N,c}} A(t)\frac{(t-rn+1)_{rn}}{(t+1)_{n+1}^S} \frac{\pi}{\sin(\pi t)} (-\alpha)^t dt
\\
&= \sum_{j=1}^p \frac{n!^{S-r}}{2i\pi}\int_{\mathcal{R}_{N,c}} B_j(t)\frac{(t-rn+1)_{rn}}{(t+1)_{n+1}^S} \frac{\pi}{\sin(\pi t)} (-\alpha)^t dt.
\end{align*}
Here we take  $\log(-\alpha)$ such that $-\pi < \arg(-\alpha) - \arg(\xi_j) \leq \pi$, where $\arg(\xi_j)$ has been chosen at the beginning of \S \ref{sec:asympT}.
Now, if  $0<\vert \alpha\vert <R$  then 
$$
\lim_{N\to +\infty} n!^{S-r} \sum_{k=rn}^N \frac{(k-rn+1)_{rn}}{(k+1)_{n+1}^S} A_k \alpha^k =T_{S,r,n}(1/\alpha)
$$
while  
\begin{align*}
\lim_{N\to + \infty}\frac{n!^{S-r}}{2i\pi}&\int_{\mathcal{R}_{N,c}} B_j(t)\frac{(t-rn+1)_{rn}}{(t+1)_{n+1}^S} \frac{\pi}{\sin(\pi t)} (-\alpha)^t dt
\\ 
&=
\frac{n!^{S-r}}{2i\pi}\int_{cn+i\infty}^{cn-i\infty} B_j(t)\frac{(t-rn+1)_{rn}}{(t+1)_{n+1}^S} \frac{\pi}{\sin(\pi t)} (-\alpha)^t dt
\\
&=\frac{(-1)^{rn}n!^{S-r}}{2i\pi}\int_{cn+i\infty}^{cn-i\infty} B_j(t)\frac{(-t)_{rn}}{(t+1)_{n+1}^S} \frac{\pi}{\sin(\pi t)} (-\alpha)^t dt
\\
&=\frac{(-1)^{rn-1}n!^{S-r}}{2i\pi}\int_{cn+i\infty}^{cn-i\infty} B_j(t)\frac{\Gamma(rn-t)\Gamma(t+1)^{S}}{\Gamma(t+n+2)^S\Gamma(-t)}\Gamma(-t)\Gamma(t+1) (-\alpha)^{t} dt
\\
&=\frac{(-1)^{rn}n!^{S-r}n}{2i\pi}\int_{c-i\infty}^{c+i\infty} B_j(tn)\frac{\Gamma((r-t)n)\Gamma(tn+1)^{S+1}}{\Gamma((t+1)n+2)^S}(-\alpha)^{tn} dt.
\end{align*}
This concludes the proof of Lemma \ref{lem:Tserint}.
\end{proof}

\subsection{Asymptotic expansion of $\grosb$}

We want to estimate these integrals using the saddle point method. We first recall 
Stirling's formula 
$$
\Gamma(z)=z^{z-1/2}e^{-z} \sqrt{2 \pi} \Big(1+\mathcal{O}\Big(\frac1z\Big)\Big),  \quad z\to \infty,
$$
valid if $| \arg (z) | \leq \pi-\eps$ with $\eps>0$; here the constant implied in $\mathcal{O}\big(\frac1z\big)$ depends on $\eps$ but not on $z$.
 By Lemma~\ref{lem:new1}, we have  
$$\grosb 
= (2\pi)^{(S-r+2)/2}\kappa_j\cdot \frac{\log(n)^{s_j}}{n^{(S+r)/2+\beta_j }}\int_{c-i \infty }^{c+i \infty }
g(t) e^{n\varphi(-\alpha/\xi_j,t)}\left(1+\mathcal{O}\Big(\frac{1}{\log(n)}\Big)\right)dt$$ 
where the constant in $\mathcal{O}$ is uniform in $t$, and  
$$
g(t)= t^{- \beta_j - (S+1)/2}(t+1)^{-3S/2}(r-t)^{-1/2} ,
$$
$$
\varphi(z,t) = t \log(z) +(S+1)t\log(t)+(r-t)\log(r-t)-S(t+1)\log(t+1).
$$
We shall be interested only in the case where $z = -\alpha/\xi_j$, but from now on we consider any non-zero complex number $z$ such that $|z| < 1$ and 
  $-\pi < \arg (z) \leq \pi$.
 Indeed we  have 
$0<\vert -\alpha/\xi_j\vert<1$ because $0<\vert \alpha\vert<R$ the radius of convergence of $F$, which is equal to the minimal value of 
$\vert \xi_j\vert$, $j=1, \ldots, p$, and letting $\log(-\alpha/\xi_j) = \log(-\alpha) - \log (\xi_j)$ yields $\arg (-\alpha/\xi_j) \in (-\pi,\pi]$ (recall that $\arg(-\alpha)$ has been chosen in the proof of Lemma \ref{lem:Tserint}).

If $-\pi < \arg (z) < \pi$, we work in the cut plane $\Omega = \C\setminus ((-\infty,0]\cup [r, +\infty))$, so that any $t\in\Omega $ is such that 
  $\arg(t)$, $\arg(t+1)$ and $\arg(r-t)$ belong to $(-\pi,\pi)$. 
On the other hand, if $z$ is real and negative (i.e., $\arg (z) = \pi$), we work in   $\Omega = \C\setminus ((-\infty,0]\cup  ( r  +e^{i\pi/8}\R_+ ))$; if $t$ is real and $0< t<r$ we take $ \arg(t) = \arg(t+1) = \arg(r-t)
 = 0$, and we use analytic continuation to define $\arg (t)$, $\arg(t+1)$ and $ \arg(r-t)$ for any $t\in\Omega$. 
 
 In both cases,  
the   function $t\mapsto \varphi(z,t)$ is analytic on the cut plane $\Omega$.
 In what follows, $\varphi'(z,t)$ and  $\varphi''(z,t)$ denote the first and second derivatives of $\varphi(z,t)$ with respect to $t$. 
We denote by  $\tau_{S,r}(z)$   the unique solution (in $t$) of the equation 
$
zt^{S+1}=(r-t)(t+1)^S
$
which is such that $\textup{Re}(\tau_{S,r}(z))>0$. (A more precise localization is given below.)  
 For simplicity, we set  $\tau_j=\tau_{S,r}(-\alpha/\xi_j)$,  $\varphi_j=\varphi(-\alpha/\xi_j,\tau_j)$,  $\psi_j=\varphi''(-\alpha/\xi_j,\tau_j)$ and $g_j=g(\tau_j)$.

\begin{lem}\label{lem:new2} Let us assume that $r=r(S)$ is an increasing function of $S$ such that $r=o(S)$ and $Se^{-S/r}=o(1)$ as $S\to +\infty$. Then if $S$ is large enough (with respect to the choice of the function $S\mapsto r(S)$), the following estimate holds:
for any $j=1, \ldots, p$, we have $\kappa_jg_j\psi_j\neq 0$ and,  
as $n\to +\infty$, 
$$\grosb = (2\pi)^{(S-r+3)/2 }\frac{\kappa_j g_j}{\sqrt{- \psi_j}}\cdot \frac{\log(n)^{s_j}e^{\varphi_j n}}{n^{( S+r+1)/2+ \beta_j}} \cdot 
\big(1+o(1)\big).$$
\end{lem}
Any choice of the form  $r(S)=[\frac{S}{\log(S)^{1+\eps}}]$ with   $\eps>0$ satisfies $r=o(S)$ and $Se^{-S/r}=o(1)$ (but not with $ \eps=0$); in Lemma \ref{lem:7nv}, we take $\eps=1$.

Note that we have three trivially equivalent expressions for $e^{\varphi_j}$:
\begin{equation} \label{eqtrois}
e^{\varphi_j}=\frac{(r-\tau_j)^{r}}{(\tau_j+1)^S}= \frac{\big((-\alpha/\xi_j)\tau_j^{S+1}\big)^r}{(\tau_j+1)^{S(r+1)}}
=-\frac{\xi_j(r-\tau_j)^{r+1}}{\alpha\tau_j^{S+1}}.
\end{equation}
\begin{proof} 
We split the proof in several steps. The assumptions made on $r$ and $S$ are not always necessary  at each step. We will write $\tau$ for $\tau_{S,r}(z)$ when there will no ambiguity.

\medskip

\noindent {\bf Step 1.} We want to begin localizing the solutions of the equation $\varphi'(z,t)=0$ (for any fixed  $z$ such that $0 < \vert z\vert <1$ and $-\pi < \arg (z) \leq \pi$), i.e. of 
$$
\log(z)+(S+1)\log(t)-\log(r-t)-S\log(t+1)=0.
$$
 These solutions are obviously amongst the solutions of the polynomial equation $P(t)=0$ where 
$$
P(t)=zt^{S+1}-(r-t)(t+1)^S.
$$
In this step, we prove the following facts: {\em For any $1\le r\le S$, the polynomial $P(t)$ has exactly $S$ roots in the half-plane $\textup{Re}(t)<-\frac12$ and one root in the half-plane  $\textup{Re}(t)>\frac12$.} 

Let us prove that there is no root in the strip $-\frac12\le \textup{Re}(t)\le\frac12$. 
We set $t=x+iy$ and assume that $-\frac12\le x\le \frac12.$ We have 
\begin{align*}
\vert t+1\vert &= \sqrt{(x+1)^2+y^2}\ge \sqrt{1/4+y^2}
\\
\vert r-t\vert &=\sqrt{(r-x)^2+y^2} \ge \sqrt{(r-1/2)^2+y^2}\ge \sqrt{1/4+y^2}
\\
\vert t\vert &=\sqrt{x^2+y^2}\le \sqrt{1/4+y^2}.
\end{align*}
Since $\vert z\vert<1$, it follows that 
$
\vert z\vert \vert t\vert^{S+1}<\sqrt{1/4+y^2}^{S+1}\le \vert r-t\vert \vert t+1\vert^S 
$
for any $t$ in the strip, which proves the claim.

Let us now prove that there are exactly $S$ roots in $\textup{Re}(t)<-\frac12$. With $u=1/t$, this amounts to prove that the equation 
$
z=(ru-1)(u+1)^S
$
has exactly $S$ solutions in the open disk $\vert u+1\vert <1$. Let us define 
$$
f(u)=z-r(u+1)^{S+1}+(r+1)(u+1)^S, \quad g(u)=z+(r+1)(u+1)^{S}.
$$
We have $f(u)-g(u)=-r(u+1)^{S+1}$ so that on the circle $\vert u+1\vert =1$ we have 
$$
\vert f(u)-g(u)\vert =r<r+1-\vert z\vert \le \vert g(u)\vert.
$$
Hence, by Rouch\'e's theorem, the equation $f(u)=0$ has the same number of solutions as $g(u)=0$ inside the disk $\vert u+1\vert <1$. There are $S$ such solutions because the solutions of $g(u)=0$ are $-1+ ( - z/(r+1))^{1/S}e^{2i\pi k/S}$, $k=0, \ldots, S-1$, which are all inside the disk.

It follows that $P(t)$ has exactly one root in the half-plane $\textup{Re}(t)>\frac12$. (We can be more precise. Let us define the  functions $P(t)=zt^{S+1}-(r-t)(t+1)^S$ and $Q(t)=-(r-t)(t+1)^S$. On the circle $\vert r-t\vert=\frac{r^2}{S+r}$, we have $\vert P(t)-Q(t)\vert = \vert z t^{S+1}\vert <\vert Q(t)\vert$. Hence, $P(t)$ has a root inside the disk $\vert r-t\vert<\frac{r^2}{S+r}$. This estimate holds for any $r,S$, but we will prove and use a more precise one under a more restrictive condition on $r$.)

\medskip

\noindent {\bf Step 2.} We need a more precise estimate for $\tau = \tau_{S,r}(z)$ that the mere fact that  $\vert \tau-r\vert<\frac{r^2}{S+r}$, namely
\begin{equation} \label{eqordd}
 \tau_{S,r}(z) =r-rz\Big(\frac{r}{r+1}\Big)^S \big(1+o(1)\big).
\end{equation}
To prove this, we consider  the power series
$$
\upsilon_{S,r}(z):= \frac{1}r - \sum_{m=1}^\infty
 \frac{\binom{(S+1)m-1}{m}}{(S+1)m-1}\frac{r^{Sm-1}}{(r+1)^{(S+1)m-1}} (-z)^m .
$$
We shall prove that it 
has radius of convergence $\frac{S^S(r+1)^{S+1}}{r^r(S+1)^{S+1}} \ge 1$, with equality only for $r=S$, and that if $r$ is an   increasing function of $S$ such that $r=o(S)$ as $n\to +\infty$ then,  provided $S$ is large enough (with respect to the choice of $r(S)$), we have $1/\upsilon_{S,r}(z)=\tau_{S,r}(z)$, the unique root of $P(t)$ in the half-plane $\textup{Re}(t)>\frac12$.

As in the first step, we solve the equation $z=V(u)$, with $V(u)=(ru-1)(u+1)^S$, and then get the solutions of $P(t)=0$ by $t=1/u$. By Lagrange's inversion formula~\cite[p. 250]{dieudonne}, a solution of the equation $z=V(u)$ is 
\begin{align*}
\frac{1}{r}+\sum_{m=1}^{\infty} \frac{1}{m!}\left(\Big(\frac{u-1/r}{V(u)-V(1/r)}\Big)^m\right)_{u=1/r}^{(m-1)} z^m
&=\frac{1}{r}+\sum_{m=1}^{\infty} \frac{r^{-m}}{m!}\Big(\frac{1}{(u+1)^{Sm}}\Big)_{u=1/r}^{(m-1)} z^m
\\
&=\frac{1}r - \sum_{m=1}^\infty
 \frac{\binom{(S+1)m-1}{m}}{(S+1)m-1}\frac{r^{Sm-1}}{(r+1)^{(S+1)m-1}} (-z)^m
\\
&=\upsilon_{S,r}(z).
\end{align*}
Since 
$$
\lim_{m\to +\infty } \Big(\frac{\binom{(S+1)m-1}{m}}{(S+1)m-1}\frac{r^{Sm-1}}{(r+1)^{(S+1)m-1}}\Big)^{1/m} 
= \frac{r^S(S+1)^{S+1}}{S^S(r+1)^{S+1}}\le 1 
$$
with equality only for $r=S$, the assertion on the radius of convergence follows. 

Since $\binom{(S+1)m-1}{m}\le S(\frac{(S+1)^{S+1}}{S^S})^{m-1}$ and $\frac{S}{(S+1)m-1}\le \frac 1m$, for any $z$ such that $\vert z\vert<1$ (inside the circle of convergence), we have 
\begin{align*}
\left\vert \upsilon_{S,r}(z)-\frac{1}{r}-\frac{z}{r}\Big(\frac{r}{r+1}\Big)^S \right\vert &\le \frac{(r+1)S^{S}}{r(S+1)^{S+1}}
\sum_{m=2}^\infty \frac1m\Big(\frac{r^S(S+1)^{S+1}}{S^S(r+1)^{S+1}} \vert z\vert\Big)^m 
\\
&\le \frac{\vert z\vert }{r}\Big(\frac{r}{r+1}\Big)^S
\Big\vert\log\Big(1-\frac{r^S(S+1)^{S+1}}{S^S(r+1)^{S+1}}\vert z\vert\Big)\Big\vert.
\end{align*}
Hence, for any $\vert z\vert <1$, $1\le r\le S$, 
$$
r\upsilon_{S,r}(z)=1+z\Big(\frac{r}{r+1}\Big)^S \left(1+\theta 
\Big\vert\log\Big(1-\frac{r^S(S+1)^{S+1}}{S^S(r+1)^{S+1}}\vert z\vert\Big)\Big\vert \right)
$$
for some $\theta$ (depending on $S,r,z$) such that $\vert \theta\vert \le 1$.

We now choose  $r$ as any fixed increasing function of $S$  such that $r=o(S)$ as $S\to +\infty$. Then
$ 
 \frac{r^S(S+1)^{S+1}}{S^S(r+1)^{S+1}} \vert z \vert  
$ 
tends to $0$ as $S\to\infty$, so that 
$$
\upsilon_{S,r}(z)= \frac{1}{r}+\frac{z}r\Big(\frac{r}{r+1}\Big)^S (1+  o(1)).
$$
Therefore, $$
\frac{1}{\upsilon_{S,r}(z)} =r-rz\Big(\frac{r}{r+1}\Big)^S (1+  o(1)). 
$$
Since $\vert z\vert <1$, the real part of $1/\upsilon_{S,r}(z)$ is  positive  for any   $S$ sufficiently large (with respect to the choice of $r(S)$) and thus $1/\upsilon_{S,r}(z)$ coincides with $\tau_{S,r}(z)$. This concludes the proof of \eqref{eqordd}.

\medskip

\noindent {\bf Step 3.} We now prove that $\tau = \tau_{S,r}(z)$  belongs to the cut plane $\Omega$ and is indeed a solution of the equation $\varphi'(z,t)=0$, provided $r$ is any fixed increasing function of $S$ such that $r=o(S)$ and $S e^{-S/r} = o(1)$  as $n\to +\infty$,  and $S$ is large enough (with respect to the choice of $r(S)$). Since $\exp (\varphi'(z,\tau)) = \frac{z\tau^{S+1}}{(r-\tau)(\tau+1)^S} = 1$, we have $ \varphi'(z,\tau) \in 2i \pi \Z$ and 
$$
\frac1i  \varphi'(z,\tau)
 =\arg(z)+ (S+1)\arg(\tau)-\arg(r-\tau)-S\arg(\tau+1) .
$$
Since $r=o(S)$ and $Se^{-S/r}=o(1)$, we have
$r-rz(\frac r{r+1})^S(1+o(1)) =r(1 + \mathcal{O}( e^{-S/r}))  $ and \eqref{eqordd} yields:  
\begin{align*}
(S+1)\arg(\tau)&=(S+1)\arg(r)+\mathcal{O}(Se^{-S/r})=o(1), 
\\
S\arg(\tau+1)&=S\arg(r+1)+\mathcal{O}(Se^{-S/r})=o(1).
\end{align*}
Moreover 
$$\arg(r-\tau)=\arg\Big(rz\big(\frac{r}{r+1}\big)^S\Big)+o(1)=\arg(z)+o(1)
$$ 
since the cut we have made on $\arg(r-t)$ is not for $\arg(r-t) = \arg(z) \bmod 2\pi$ (here we use the alternative definition of $\Omega$ when $\arg (z) = \pi$, intended to have $\frac{-7\pi}8 < \arg (r-t) < \frac{9\pi}8$  in this case).  Therefore $\tau\in\Omega$ provided $S$ is large enough. Moreover
  $\frac1i  \varphi'(z,\tau)$ tends to 0 as $S\to\infty$, and belongs to $ 2 \pi \Z$: it is 
   $0$ if $S$ is large enough with respect to the choice of $r(S)$.

\medskip

\noindent {\bf Step 4.} We now prove that $g(\tau_{S,r}(z))\neq 0$ and $\varphi''(z,\tau_{S,r}(z))\neq 0$  provided $r$ is any fixed increasing function of $S$ such that $r=o(S)$ as $n\to +\infty$, and $S$ is large enough (with respect to the choice of $r(S)$). 

Since $(\frac{r}{r+1})^S = o(1)$, \eqref{eqordd}   yields
\begin{equation}\label{eq:608}
g(\tau)=  \frac{1}{\tau^{\beta_j+(S+1)/2}(\tau+1)^{3S/2}(r-\tau)^{1/2}}  =
\frac{1+o(1)}{z^{1/2}  r^{\beta_j+S+1} (r+1)^{S }} 
\end{equation}
and 
\begin{equation}
\varphi''(z,\tau) = \frac{S+1}{\tau} + \frac{1}{r-\tau} -\frac{S}{\tau+1} 
= \frac{(r+1)^S}{r^{S+1}z} \big(1+o(1)\big) \label{eq:609}
\end{equation}
provided $Se^{-S/r}=o(1)$ for~\eqref{eq:609}. 
The right-hand sides of~\eqref{eq:608} and \eqref{eq:609} are both non-zero if $S$ is large enough with respect to the choice of $r(S).$

\medskip

\noindent {\bf Step 5.} In this step we choose $r(S)$ as in the statement of the lemma, so that all the previous steps are simultaneously valid provided $S$ is large enough with respect to the choice of $r(S)$.   We want to determine an {\em admissible path} passing through $\tau_{S,r}(z)$, i.e. a path  along which $t\mapsto \textup{Re}(\varphi(z, t))$ has  a unique global maximum at $t=\tau_{S,r}(z) $. This determination process is rather lengthy  as we have to consider three cases:  $\textup{Re}(z)> 0$, $\textup{Re}(z)<0$ and $\textup{Im}(z)\neq 0$, and $\textup{Re}(z)<0$ and $\textup{Im}(z)=0$. 
Note that similar computations are done in~\cite{nishimoto, zudilin} for the same kind of purpose.  In particular, analogues of the contours $\chemz $, $\chemti $ and $\chemhat $ constructed below are also considered in these papers.
Throughout the computations we always assume $S$ to be sufficiently large. 

\medskip

\noindent $\bullet$ {\bf Case $\textup{Re}(z)> 0$.}  \eqref{eqordd} yields   $0<\textup{Re}(\tau)<r$ so that the  vertical line $\chemz$ passing through $\tau$ is inside the strip $0< \textup{Re}(t)<r$ and we are going to prove that it is admissible.  

We set $\uuu=\textup{Re}(\tau)$,  $\chemz=\{\uuu+iy, y\in \mathbb R\}$ and $w_0(y)=\textup{Re}\big(\varphi(z,\uuu+iy)\big)$. We have
$$
w_0'(y)=-\textup{Im}\big(\varphi'(z, \uuu+iy)\big) = -\arg(z)-(S+1)\arg(\uuu+iy)+\arg(r-\uuu-iy)+S\arg(1+\uuu+iy).
$$
Hence
\begin{equation} \label{eqlim}
\lim_{y\to -\infty} w_0'(y)=\pi-\arg(z)\ge 0, \quad \lim_{y\to +\infty} w_0'(y)=-\pi-\arg(z)\le 0.
\end{equation}
Moreover  $\textup{Re}(r-\uuu-iy)=r(\frac{r}{r+1})^S\textup{Re}(z)(1+o(1))>0$, $\textup{Re}(\uuu+iy)=r(1+o(1))>0$, $\textup{Re}(\uuu+1+iy)=(r+1)(1+o(1))>0$ so that
$$
w_0'(y)=-\arg(z)-(S+1)\arctan\Big(\frac{y}{\uuu}\Big)-\arctan\Big(\frac{y}{r-\uuu}\Big)+S\arctan\Big(\frac{y}{1+\uuu}\Big)
$$
and 
\begin{align*}
w_0''(y)&= -\frac{(S+1)\uuu}{\uuu^2+y^2} -\frac{r-\uuu}{(r-\uuu)^2+y^2} +S\frac{(1+\uuu)}{(1+\uuu)^2+y^2}
\\
&=\frac{- N(y^2)}{(\uuu^2+y^2)((r-\uuu)^2+y^2)((1+\uuu)^2+y^2)} 
\end{align*}
upon letting   $N(x)=ax^2+bx+c$ where
$$
a=r-S<0 , \quad b= -Sr^2+r^2\uuu+2Sr\uuu+2r\uuu+S\uuu+r, 
$$
$$
c= \uuu(1+\uuu)(r-\uuu)(Sr+r\uuu-S\uuu+r)>0.
$$
The equation $N(x)=0$ as a negative root (because $ac<0$) and another one asymptotically equal to 
$ r^2\big(1+o(1)\big)$.
This root is $> \textup{Im}(\tau)^2$  because \eqref{eqordd} yields
$
\textup{Im}(\tau)= o(1).
$
Hence 
$w_0''(y)=0$ has exactly two solutions: a positive and negative one, with $\textup{Im}(\tau)$  strictly in between. Since $w_0'(\textup{Im}(\tau)) =0$, \eqref{eqlim} ensures 
 that $w_0'(y)$ vanishes   at $\textup{Im}(\tau)$,  is positive for $y<\textup{Im}(\tau)$ and negative for $y>\textup{Im}(\tau)$. Hence $w_0(y)$ is maximal at $\textup{Im}(\tau)$; this completes the proof of this case.

\medskip

\noindent $\bullet$ {\bf Case $\textup{Re}(z)<0$ and $\textup{Im}(z)\neq 0$.} In this case, the  vertical  line passing through $\tau$  is no longer inside the strip $0< \textup{Re}(t)<r$ and we have to deform it. We assume that $\textup{Im}(z)<0$, the other case being delt with similarly; then  $\textup{Im}(\tau)>0$. 

We first want to determine a segment passing through $\tau$ along which  $t\mapsto \textup{Re}(\varphi(z,t))$ admits a local maximum at 
 $t=\tau$. Let $\beta=\arg(\varphi''(z,\tau))\in (-\pi,\pi]$ and $z=\rho e^{i\delta}$ with $\rho > 0$ and $\delta\in (-\pi,-\pi/2)$ because  $\textup{Re}(z)<0$ and $\textup{Im}(z)<0$. Then, from \eqref{eq:609} in Step 4, we have
$$
\varphi''(z,\tau) =\frac{1}{r\rho} \Big(\frac{r+1}{r}\Big)^S e^{-i\delta} \big(1+o(1)\big)
$$
so that $\beta= -\delta+o(1)$. Therefore any $\theta\in\R$ such that $\cos(2\theta-\delta) < 0$ satisfies also   $\cos(2\theta+\beta)<0$ provided $S$ is large enough (in terms of $\theta$); then $\theta$ is said to be admissible. Obviously $\theta=0$ and any $\theta$ sufficiently close to $\pi+\delta$ are admissible.  
 By the theory of steepest paths of analytic functions (see \cite[pp. 255--258]{blehan}), for any admissible $\theta$ there exists $\eta>0$ such that the function $t\mapsto \textup{Re}(\varphi(z,t))$ admits a unique global maximum at $t=\tau$ where 
$t$ is on the segment $\{\tau+e^{i\theta}y$, $\vert y\vert\le \eta\}$. 

This suggests to define  a polygonal path $\chemti$ as the union $\chemti=\chemu\cup \chemd\cup \chemt$ where $\chemu=\{r-iy, y\ge 0\}$, $\chemd=[r, \tau]$ and $\chemt=\{\tau+y, y\ge 0\}$: $\chemu$ is a  vertical  half-line, $\chemd$ is a segment and $\chemt$ is an   horizontal  half-line.  We claim that $t\mapsto \textup{Re}(\varphi(z,t))$ admits a unique global maximum at $t=\tau$ when $t$ varies in $\chemti$; this function is continuous on $\chemti$ and can be differentiated on  $\chemti \setminus \{r,\tau\}$. 

First,  $w_1(y)=\textup{Re}(\varphi(z, r-iy))$ is decreasing on $[0,+\infty)$ since for any $y>0$: 
\begin{align*}
w_1'(y)=\textup{Im}\big(\varphi'(z, r-iy)\big)
&=\arg(z)+(S+1)\arg(r-iy)-\arg(iy)-S\arg(1+r-iy)
\\
&=\arg(z)-(S+1)\arctan\Big(\frac{y}r\Big)-\frac{\pi}{2}+S\arctan\Big(\frac{y}{r+1}\Big)< 0
\end{align*}
because $\arg(z)\le -\frac{\pi}{2}$.

Let us now prove that $w_3(y)=\textup{Re}(\varphi(z,\tau+y))$ is decreasing on $ [0,+\infty)$. We have  
$$
w_3'(y)=\textup{Re}\big(\varphi'(z,\tau+y)\big)
 =\log\left\vert \frac{z(\tau+y)^{S+1}}{(r-\tau-y)( \tau+y+1)^S}\right\vert \neq 0 
$$
for any $y>0$ using Step 1: 
 the only $t$  in $\textup{Re}(t)\ge 0$ such that $\frac{zt^{S+1}}{(r-t)(t+1)^S}=1$ is $t=\tau$. Therefore $w_3$ is monotonic; since $\theta=0$ is admissible it is decreasing. 

It remains to prove that $t\mapsto \textup{Re}(\varphi(z,t))$ admits a unique global maximum at $t=\tau$ when $t$ varies in $\chemd$. We parametrize 
$\chemd$ as $\{\tau+ye^{i\gamma}, y \in[  0, \vert  U\vert ]\}$ with (by definition) $U = r- \tau $ and $\gamma=  \arg(  U) = \delta+o(1)$ using \eqref{eqordd}; then $\gamma\in (-\pi, -\pi/2)$. Let $ w_2(y)=\textup{Re}(\varphi(z,\tau+ye^{i\gamma}))$. Then
$$w_2'(y) =\cos(\gamma) \textup{Re}\big(\varphi'(z,\tau+ye^{i\gamma})\big)-\sin(\gamma)\textup{Im}\big(\varphi'(z,\tau+ye^{i\gamma})\big).$$
The function $\ell(y)=  \textup{Re}\big(\varphi'(z,\tau+ye^{i\gamma})\big) = \log\left\vert \frac{z(\tau+ye^{i\gamma})^S}{(r-\tau-ye^{i\gamma})(\tau+1+ye^{i\gamma})^S}\right\vert$ satisfies $\ell(0)=0$ and $\ell(y)\neq 0$ for any $y \in [ 0, \vert  U\vert)$ (using Step 1 again); moreover $
\lim_{y\to  \vert  U\vert} \ell(y)=+\infty
$. Therefore we have $\ell(y)>0$ for any $y \in[0, \vert  U\vert)$.  
We now analyse the term $a(y)  = \textup{Im}\big(\varphi'(z,\tau+ye^{i\gamma})\big)$; we have
\begin{align*}
a(y)&=\arg(z)+(S+1)\arg(\tau+ye^{i\gamma})-S\arg(\tau+1+ye^{i\gamma})-\arg(r-\tau-ye^{i\gamma})
\\
&= \delta + (S+1) ( \gamma+\arg( \tau e^{-i\gamma}+y)) - S( \gamma+\arg((\tau+1)e^{-i\gamma}+y)) - \arg( |U| e^{i\gamma}-y  e^{i\gamma}) \\
& =\pi+\delta-(S+1)\arctan\left(\frac{r\sin(\gamma)}{r\cos(\gamma)+y - \vert  U\vert }\right)+S\arctan\left(\frac{(r+1)\sin(\gamma)}{(r+1)\cos(\gamma)+y - \vert U\vert }\right)
\end{align*}
since for $\zeta = \tau e^{-i\gamma}+y$ we have $\arg(\zeta) = \pi + \arctan(\frac{\Im (\zeta)}{\Re (\zeta)})$.  
This function is decreasing on $[0, \vert  U\vert]$ because 
\begin{eqnarray*}
a'(y)&=&(S+1)\frac{r \sin (\gamma)}{r^2\sin^2(\gamma)+(r\cos(\gamma)+y -  \vert U\vert  )^2}\\
&& \quad \quad \quad \quad -S\frac{(r+1) \sin (\gamma)}{(r+1)^2\sin^2(\gamma)+((r+1)\cos(\gamma)+y -  \vert  U\vert )^2}\\
&\leq & \frac{(S+1)r\sin \gamma}{(r+|U|)^2} -  \frac{ S (r+1) \sin \gamma}{(r+1)^2} < 0
\end{eqnarray*}
since $\frac{r}{(r+|U|)^2} > \frac1{r+1}$ because $|U|<\frac14$ (using Step 2). 
In Step 3, we proved that $a(0)=0$, so that $a(y)\leq 0$ for any $y \in[0, \vert  U\vert]$.  
  It follows that   
$$
 w_2'(y) = \cos(\gamma)\ell(y)-\sin(\gamma)a(y) <  0
$$
for any $y\in [0, \vert  U\vert]$.

We have thus proved that $t\mapsto \textup{Re}(\varphi(z,t))$ admits a unique global maximum at $t=\tau$ when $t$ varies in $\chemti$. We cannot integrate directly over $\chemti$ because $r$ is a singularity of $\varphi(z,\cdot)$. Hence, we slightly deform $\chemti$ around the 
``corner'' of the path at $r$: we replace that corner with an arc of circle of center $r$ and small positive radius $\kappa$, in which $\arg(r-t)$ varies in $[\gamma, \pi/2]$. We 
connect this arc with the remaining parts of $\chemu$ and $\chemd$, and with $\chemt$,  to get a new path $\chemti'$. By continuity of $t \mapsto\textup{Re}(\varphi(z,t))$ in this region, we can take $\kappa$ small enough so that it still admits a unique global maximum at $t=\tau$ when $t$ varies in $\chemti '$.

\medskip

\noindent $\bullet$ {\bf Case $\textup{Re}(z)<0$ and $\textup{Im}(z)= 0$.} In this case, $\tau$ is a real number greater than $r$. As in the previous case we obtain $\arg(\varphi''(z,\tau))=-\pi+o(1) \bmod 2\pi$: the  angles $\theta$ such that   $\cos(2\theta-\pi)<0$ are admissible, for instance $\theta=0$.  This suggests to define a polygonal path $\chemhat$ as the union $\chemhat=\chemq\cup \chemc$ where $\chemq=\{r+iy, y\ge 0\}$ and $\chemc=\{r+y, y\ge 0\}$. Since $\Omega = \C\setminus ((-\infty, 0)\cup (r + e^{i\pi/8}\R_+))$ in the present case, $\chemq$ and $  \chemc$  are contained in $\Omega\cup\{r\}$. We claim that $t\mapsto \textup{Re}(\varphi(z,t))$ admits a unique global maximum at $t=\tau>r$ when $t$ varies in $\chemhat$. 

Letting $w_5(y) = \varphi(z,y)$ we obtain (as for $w_3$  in the previous case) that $w_5'(y)$ vanishes at $y=\tau$, is positive for $r < y < \tau$ and negative for $y>\tau$. 
Hence, $y\mapsto \textup{Re}(\varphi(z, y))$ admits a unique maximum on $[r, +\infty)$ achieved at 
$y=\tau$. Thus to prove the claim, it remains to prove that  $w_4(y)=\textup{Re}(\varphi(z, r+iy))$ is decreasing on $[0,+\infty)$. Now, as for $w_0$ in the case $\Re(z)>0$  we have for any $y>0$:
\begin{align*}
 w_4'(y)  
&=-\arg(z)-(S+1)\arg(r+iy)+\arg(-iy)+S\arg(1+r+iy)\\
&=-\frac{3\pi}2-(S+1)\arctan\Big(\frac{y}r\Big)+S\arctan\Big(\frac{y}{r+1}\Big) < 0
\end{align*}
and the claim is completely proved. Again, we cannot integrate directly over $\chemhat$ because $r$ is a singularity of $\varphi(z,\cdot)$. Hence, we slightly deform $\chemhat$ around the ``corner'' of the path at $r$: we replace that corner with an arc of circle of center $r$ and small positive radius $\kappa$, contained in the cut plane $\Omega$. We 
connect this arc with the remaining parts of $\chemq$ and $\chemc$ to get a new path $\chemhat'$. By continuity of $t \mapsto\textup{Re}(\varphi(z,t))$ in this region, we can take $\kappa$ small enough and ensure that it still admits a unique global maximum at $t=\tau$ when $t$ varies in $\chemhat'$.

\bigskip

\noindent {\bf Step 6.} We are now  in position to conclude the proof of Lemma~\ref{lem:new2}. We recall that 
$$\grosb 
= (2\pi)^{(S-r+2)/2}\kappa_j\cdot \frac{\log(n)^{s_j}}{n^{(S+r)/2+\beta_j }}\int_{c-i \infty }^{c+i \infty }
g(t) e^{n\varphi(-\alpha/\xi_j,t)}\left(1+\mathcal{O}\Big(\frac{1}{\log(n)}\Big)\right)dt$$ 
where the constant in $\mathcal{O}$ is uniform in $t$. Depending on the location of $-\alpha/\xi_j$ in the open unit disk (with respect to the three cases in Step 5), we move the integration path from the vertical line $\Re(t) = c$ to the path $\chemz$, $\chemti '$ or $\chemhat'$ where the orientation is from $\textup{Im}(t)\leq 0$ to $\textup{Im}(t)>0$. In the previous steps, we have done everything to ensure that the saddle point method (see~\cite[Chapitre IX]{dieudonne} or \cite[Proposition~7]{firi}) can be applied to this path and we get  
$$
\int_{c-i \infty }^{c+i \infty }
g(t) e^{n\varphi(-\alpha/\xi_j,t)}\left(1+\mathcal{O}\Big(\frac{1}{\log(n)}\Big)\right)dt = 
g_j \cdot \sqrt{\frac{2\pi}{- n\psi_j}}\cdot e^{\varphi_j n}\cdot \big(1+o(1)\big)
$$ 
provided $r(S)$ is chosen as in the statement of Lemma~\ref{lem:new2}  and   $S$ is large enough. This concludes the proof of 
  Lemma~\ref{lem:new2}, since $\kappa_j\neq 0$ (using Lemma \ref{lem:new1}).
 \end{proof}

\subsection{Asymptotic behavior of $T_{S,r,n}(1/\alpha)$}

We now state our final resut, the first part of which immediately comes from combining Lemmas~\ref{lem:Tserint} and \ref{lem:new2}. Recall that $s_j$, $\beta_j$, $\kappa_j$ have been defined in  Lemma \ref{lem:new1}, $g_j$, $\varphi_j$, $\psi_j$ just before  Lemma~\ref{lem:new2}.

\begin{prop}\label{prop:asympT} Let us assume that $0<\vert \alpha\vert <R$, and  $r=r(S)$ is an increasing function of $S$ such that $r=o(S)$ and $Se^{-S/r}=o(1)$ as $S\to +\infty$. Then if $S$ is large enough (with respect to the choice of the function $r(S)$), the following estimate holds:
for any $j=1, \ldots, p$, we have $\kappa_jg_j\psi_j\neq 0$ and as $n\to +\infty$
\begin{equation}\label{eq:asympT2}
T_{S,r,n}(1/\alpha)=\frac{(2\pi)^{(S-r+1)/2}(-1)^{rn}}{n^{( S+r-1)/2}}\sum_{j=1}^p \Big(\frac{\kappa_jg_j}{\sqrt{ \psi_j}}\cdot  n^{-\beta_j}\log(n)^{s_j}
e^{\varphi_j n} \cdot 
\big(1+o(1)\big)\Big).
\end{equation}
Moreover, if $r^\omega e^{-S/r}=o(1)$ for any $\omega>0$ then  the numbers $e^{\varphi_j}$ (for $j=1, \ldots, p$)  are pairwise distinct.
\end{prop}

The only new property in Proposition \ref{prop:asympT} is that the numbers $e^{\varphi_j}$ are pairwise distinct; we shall prove it below. All the conditions on $r$ are satisfied if $r=[S/\log(S)^{1+\varepsilon}]$ for any fixed $\varepsilon>0.$ 
Let us now deduce Lemma \ref{lem:7nv} (stated in \S \ref{subsec53}) from Proposition \ref{prop:asympT}. Let $a = \max(\Re(\varphi_1),\ldots, \Re(\varphi_p))$, and denote by $J$ the non-empty set of all $j\in\{1,\ldots,p\}$ such that $\Re (\varphi_j) = a$. Let $(\kappa,\lambda)$  denote the maximal value of $(-\beta_j - \frac12(S+r-1), s_j)$, $j\in J$, with respect to lexicographical order. Denote by $j_1$, \ldots, $j_Q$ (with $Q\geq 1$) the pairwise distinct elements $j\in J$ such that 
$(-\beta_j - \frac12 (S+r-1), s_j) = (\kappa,\lambda)$. Then in the sum \eqref{eq:asympT2} we may restrict to $j\in \{j_1, \ldots, j_Q\}$. The numbers $\zeta_q = (-1)^r \exp(i \Im (\varphi_{j_q}))$, $1\leq q \leq Q$, are  pairwise distinct because $ \varphi_{j_1}$, \ldots, $ \varphi_{j_Q}$ are; and the numbers $c_q = (2\pi)^{(S-r+1)/2} \kappa_{j_q}g_{j_q}/\sqrt{ \psi_{j_q}}$ are non-zero. At last, we have $0 < a := | e^{ \varphi_{j_q}}| \le \frac{r^r}{r^S}$ if $S$ is large enough, using the first expression in \eqref{eqtrois} and the fact that $\tau_{j_q}$ tends to $r$ as $S\to\infty$. This concludes the proof of Lemma \ref{lem:7nv}.

\begin{proof}
We only need to prove the assertion on the numbers $e^{\varphi_j}$. There is nothing to prove if $p=1$ and we now assume that $p\ge 2$. Letting $z_j=-\alpha/\xi_j$,  \eqref{eqordd} and the second expression of \eqref{eqtrois} yield
\begin{eqnarray}
e^{\varphi_j}
&=& \frac{z_j^r r^{r(S+1)}}{(r+1)^{S(r+1)}} \Big(\frac{\tau_j}{r}\Big)^{r(S+1)} \Big(\frac{\tau_j+1}{r+1}\Big)^{-S(r+1)} \nonumber \\
&=& \frac{z_j^r r^{r(S+1)}}{(r+1)^{S(r+1)}} \Big( 1- rz_j \Big( \frac{r}{r+1}\Big)^S \big(1+o(1)\big)\Big) \label{eqprecis}
\end{eqnarray}
where the error term $o(1)$ depends on $j$ (and tends to 0 as $S\to\infty$). Now assume that $e^{\varphi_j} =  e^{\varphi_\ell}$ with $j\neq \ell$ (so that $z_j\neq z_\ell$). Taking the limit of $|e^{\varphi_j}  (r+1)^{S(r+1)} r^{-r(S+1)}|^{1/r}$ yields $|z_j| = | z_\ell |$ (provided $S$ is large enough). Considering the next term in the expansion given by \eqref{eqprecis}, the equality $|e^{\varphi_j}|  =  |e^{\varphi_\ell}|$ then yields $\Re(z_j) = \Re(z_\ell)$, so that $z_\ell = \overline{z_j}$. This implies $\tau_\ell = \overline{\tau_j}$, and $ e^{\varphi_\ell} = \overline{ e^{\varphi_j}}$ using  \eqref{eqtrois}, so that $e^{\varphi_j} =  e^{\varphi_\ell}$ is real. Let $\theta_j = \arg(z_j)$; then  \eqref{eqprecis} yields $r\theta_j - k\pi = \mathcal{O}(r e^{-S/r})$ for some $k\in\Z$. By assumption this implies $r\theta_j - k\pi = o(r^{-\omega})$ for any $\omega>0$. However $z_j$ is algebraic, and the theory of linear forms in logarithms shows that $\theta_j /\pi$ is not a Liouville number (see for instance \cite[Chapter 4]{EMS}). Therefore  $\theta_j /\pi$ is rational, and $r\theta_j - k\pi = 0$. Using  \eqref{eqprecis} again we obtain that $z_j$ is real, so that $z_\ell =  \overline{z_j} = z_j$. This contradiction  completes the proof of Proposition \ref{prop:asympT}.
\end{proof}

\section{Remark on the case of non-negative coefficients}  \label{secrempos}

We conclude this paper with a methodological remark. The saddle point method is a very powerful and general method, but its effective implementation can be long and difficult. This is undoubtedly the case in our situation as 
 \S\ref{sec:asympT} shows. Hence,   
it is useful to have alternative methods that can be applied at least in special (and still important) cases. Such a method exists   
when $A_k\ge 0$ for all large enough $k$:   the conclusion of Theorem~\ref{theo:1} can then be obtained faster, at least if $\alpha$ is also assumed to be a positive algebraic number. For this, we use a representation of $T_{S,r,n}(z)$ as a {\em real} integral instead of the {\em complex} integral representation of  Lemma~\ref{lem:Tserint}.  In~\eqref{eq:104} and~\eqref{eq:1044} below, we make no assumption on the $A_k$'s.
\begin{prop}\label{lem:5} Let $z$ be such that $\vert z\vert >1/R$. We have 
\begin{equation}\label{eq:104}
T_{S,r,n}(z)=\frac{z^{-rn}}{n!^r}\int_{[0,1]^S} F^{(rn)}\Big(\frac{t_1\cdots t_S}z\Big) \prod_{j=1}^S t_j^{rn}(1-t_j)^n dt_j, \quad n\ge 0, 
\end{equation}
and
\begin{equation}\label{eq:1044}
\limsup_{n\to +\infty} \vert T_{S,r,n}(z)\vert^{1/n} \le \frac{1}{r^{S-r}}.
\end{equation}
Moreover, if $F$ is not a polynomial,   $z>1/R$, and $A_k\ge 0$ for all $k$ large enough, then
\begin{equation}\label{eq:10444}
\liminf_{n\to +\infty}  T_{S,r,n}(z)^{1/n} \ge \frac{1}{D^r z^r}\left(\frac{r}{r+1}\right)^{rS}\frac{1}{(r+1)^{S-r}}>0
\end{equation}
with $D$  such that $D_n \leq D^{n+1}$ for any $n$, where
$D_n$ is  the smallest positive integer such that $D_nA_k$ is an algebraic integer for any $k\le n$. 
\end{prop}
\begin{Remark} If $A_k=1$ for all $k\ge 0$, we have 
$F^{(rn)}(x)=\frac{(rn)!}{(1-x)^{rn+1}}$ and \eqref{eq:104} coincides with (1) of \cite[Lemme 1]{ribordeaux} (up   to a factor of $z$).
\end{Remark}
\begin{proof} For any $x$ such that $\vert x\vert <R$, we have
\begin{equation}\label{eq:106}
F^{(rn)}(x)=\sum_{k=0}^\infty (k-rn+1)_{rn}A_k x^{k-rn}
\end{equation}
and the series converges absolutely. Since $\vert t_1\cdots  t_S/z\vert <R$, we can thus exchange integral and summation below: 
\begin{align*}
\frac{z^{-rn}}{n!^r}\int \limits_{[0,1]^S} F^{(rn)}\Big(\frac{t_1\cdots t_S}z\Big) &\prod_{j=1}^S t_j^{rn}(1-t_j)^n dt_j
=\sum_{k=0}^\infty \frac{(k-rn+1)_{rn}}{n!^r}A_k z^{-k} \left( \int_0^1 t^{k}(1-t)^n dt \right)^S
\\
&=\sum_{k=0}^\infty \frac{(k-rn+1)_{rn}n!^Sk!^S}{n!^r(n+k+1)!^S}A_k z^{-k}.
\end{align*}
This series is nothing but $T_{S,r,n}(z)$, which proves the first part.

\medskip

As in the proof of \cite[Lemme 3]{ribordeaux}, we now observe that, for any $k\ge rn$, 
\begin{multline*}
\left \vert n!^{S-r}\frac{k(k-1)\ldots (k-rn+1)}{(k+1)^S(k+2)^S\cdots (k+n+1)^S}  \right\vert \le n^{(S-r)n} \frac{k^{rn}}{k^{S(n+1)}} 
 \le \left (\frac{n}k\right)^{(S-r)n}\frac{1}{k^S}
  \le \frac{1}{r^{(S-r)n}}\frac{1}{k^S}.
\end{multline*}
Therefore, 
$$
\vert T_{S,r,n}(z)\vert \le \frac{1}{r^{(S-r)n}}\sum_{k=rn}^{\infty}\frac{A_k\vert z\vert^{-k} }{k^S} \le \frac{1}{r^{(S-r)n}}\sum_{k=0}^{\infty}A_k\vert z\vert^{-k}, 
$$
where the series converges because $\vert z\vert >1/R$, and~\eqref{eq:1044} follows 
as claimed.

\medskip

We now assume that $A_k\ge 0$ for all $k$ large enough, and $A_k\neq 0$ for infinitely many $k$. We also assume that $z>1/R.$ We start from \eqref{eq:106} with $0<x<R$: 
\begin{align*}
\frac{1}{(rn)!}F^{(rn)}(x)&=\sum_{k=rn}^\infty \frac{(k-rn+1)_{rn}}{(rn)!}A_k x^{k-rn}=\sum_{k=0}^\infty \frac{(k+1)_{rn}}{(rn)!}A_{k+rn} x^{k}
\ge \sum_{k=0}^\infty A_{k+rn} x^{k}.
\end{align*}
Now the sequence $(A_k)$ satisfies  (for $k$ large enough) a linear recurrence of order $\ell$ (as in the proof of Step 1 of Lemma \ref{lemlinrec}, but expanding at 0 rather than $\infty$) and it is  non-zero infinitely often. Hence, in fact, for any $n$ sufficiently large, there exists $k_n\in\{0, \ldots, \ell-1\}$ such that $A_{rn+k_n}\neq 0$. In particular, $D_{rn+k_n}A_{rn+k_n}\ge 1$. 
It follows that
$
\frac{1}{(rn)!}F^{(rn)}(x) \ge A_{k_n+rn} x^{k_n} \ge \frac{x^{k_n}}{D^{rn+k_n+1}}.
$
We use this lower bound in~\eqref{eq:104} with $x=t_1\cdots t_S/z$: 
\begin{align*}
T_{S,r,n}(z)
&\ge \frac{1}{D^{\ell+rn}z^{rn}\max(1,z)^{\ell-1}}\frac{(rn)!}{n!^r} \left(\int_0^1 t^{rn+\ell-1}(1-t)^n dt \right)^S 
\\
&= \frac{1}{D^{\ell+rn}z^{rn}\max(1,z)^{\ell-1}} \frac{(rn)!n!^{S-r}(rn+\ell-1)!^S}{((r+1)n+\ell)!^S}.
\end{align*}
We then deduce \eqref{eq:10444} by Stirling's formula.
\end{proof}
With $r=[S/\log(S)^2]$, these upper and lower bounds for $T_{S,r,n}(z)$ are essentially identical when $S\to +\infty$. With $z=1/\alpha$ for some  algebraic number $\alpha$ in $(0, R)$, we can conclude directly  in \S\ref{sec:completion} with an application of T\"opfer's criterion instead of Theorem~\ref{th:nest}.

\bigskip

\noindent S. Fischler, 
Laboratoire de Math\'ematiques d'Orsay, Univ. Paris-Sud, CNRS, Universit\'e Paris-Saclay, 91405 Orsay, France.

\medskip

\noindent T. Rivoal, Institut Fourier, CNRS et Universit\'e Grenoble Alpes, CS 40700, 
 38058 Grenoble cedex 9, France

\bigskip

\noindent {\em Key words and phrases.} $G$-functions, $G$-operators, Linear independence criterion, Singularity analysis, Saddle point method.

\medskip

\noindent {\em 2010 Mathematics Subject Classification.} Primary 11J72, 11J92, Secondary 34M35, 41A60.

\end{document}